\documentclass[onefignum,onetabnum]{siamonline250211}

\usepackage{siunitx}
\usepackage{lipsum}
\usepackage{amsfonts}
\usepackage{graphicx}
\usepackage{epstopdf}
\usepackage{algorithm}
\usepackage{algpseudocode}
\usepackage{subcaption}
\usepackage{booktabs}
\usepackage{mathtools}
\usepackage{caption}
\usepackage{mathtools}
\ifpdf
  \DeclareGraphicsExtensions{.eps,.pdf,.png,.jpg}
\else
  \DeclareGraphicsExtensions{.eps}
\fi

\usepackage{enumitem}
\setlist[enumerate]{leftmargin=.5in}
\setlist[itemize]{leftmargin=.5in}

\newsiamremark{remark}{Remark}
\newsiamremark{hypothesis}{Hypothesis}
\crefname{hypothesis}{Hypothesis}{Hypotheses}
\newsiamthm{claim}{Claim}
\newsiamremark{fact}{Fact}
\crefname{fact}{Fact}{Facts}

\headers{T. Larsen and A. Alexanderian}{T. Larsen and A. Alexanderian}

\title{
A new kernel-based approach for the global sensitivity analysis of models with correlated inputs
\thanks{Submitted to the editors on \today.
\funding{This work was supported in part by NSF grant DMS-2342344.}}}

\author{Troy Larsen\thanks{Department of Mathematics, North Carolina State University, Raleigh, NC.}
\and Alen Alexanderian\footnotemark[2]}

\usepackage{amsopn}

\DeclareMathOperator{\HSIC}{HSIC}

\newcommand{\cB}{\mathcal{B}}
\newcommand{\cF}{\mathcal{F}}
\newcommand{\cG}{\mathcal{G}}
\newcommand{\cH}{\mathcal{H}}
\newcommand{\cX}{\mathcal{X}}
\newcommand{\cY}{\mathcal{Y}}

\newcommand{\bP}{\mathbb{P}}
\newcommand{\bR}{\mathbb{R}}

\newcommand{\1}{\mathbf{1}}

\newcommand{\defeq}{\vcentcolon=}

\ifpdf
\hypersetup{
  pdftitle={A new kernel-based approach for the global sensitivity analysis of models with correlated inputs},
  pdfauthor={T. Larsen and A. Alexanderian}
}
\fi

\begin{document}

\maketitle

\begin{abstract}
We present an HSIC-based approach for global sensitivity analysis of broad classes of models with correlated and possibly function-valued inputs and outputs. To this end, we define the first-order and total HSIC sensitivity indices, which are bounded, interpretable, and moment-independent analogues to the first-order and total-effect Sobol’ indices, respectively. These desirable qualities hinge upon the key property of {\em monotonicity under marginalization} for the HSIC. We rigorously establish this monotonicity property by using a suitable class of augmented kernels. Furthermore, we provide an efficient algorithm for computing an empirical estimator of the HSIC that significantly reduces computational complexity and storage requirements. The effectiveness and interpretability of the first-order and total HSIC sensitivity indices are demonstrated through computational experiments on models that feature nonlinear relationships, correlated inputs, and functional outputs.
\end{abstract}



\section{Introduction}

Global sensitivity analysis (GSA) is an aspect of uncertainty quantification
that addresses the following question:
given a mathematical model 
\begin{equation}\label{eq:model_basic}
    Y = f(X), \quad X = (X_1, \dots, X_p), 
\end{equation} 
how is the uncertainty in $Y$ apportioned to the different input parameters? In many applications, the inputs and output may be high-dimensional or function-valued. Accordingly, we
study this problem in a general functional analytic setting, where each
parameter $X_i$ and the output $Y$ are assumed to belong to real separable
Hilbert spaces. No assumption of independence is imposed on the inputs. This generality allows us to address sensitivity measures that apply to broad classes of models.

One widely-used approach to GSA utilizes the variance decomposition first introduced by
Sobol'~\cite{Sobol01,Sobol09,Sobol07}. The corresponding Sobol' indices quantify the contribution of each input or subset of inputs to the total output variance. While originally developed for scalar-valued functions of finitely many independent real-valued random inputs, subsequent works have extended the theory of the Sobol'
indices to vector-valued~\cite{Gamboa14} or 
function-valued~\cite{Alexanderian20,Cleaves19,Gamboa14} models. However, these developments fundamentally rely on the assumption of parameter independence.

Extensions of Sobol' indices to models with dependent inputs have been developed and studied as well; see, e.g.,~\cite{Hart18,CorrGSA}.  In the correlated setting, however, the Sobol' indices lose their intuitive interpretation as the contributions to the total output variance. This also complicates the use of the Sobol' indices for parameter dimension reduction by fixing the non-influential parameters at their nominal values.

These limitations have motivated the development of alternative GSA approaches 
for models with correlated inputs. These include derivative-based
methods~\cite{Kucherenko17, Lamboni21}, variance-based
analysis~\cite{Kucherenko12,CorrGSA,Rabitz10,Zhang15}, and Shapley
effects~\cite{daVeiga21, Iooss19}.  Our work builds on kernel-based
methods for GSA~\cite{Barr22,daVeiga15,daVeiga21,Durrande11,Gretton05a}, which provide
a flexible and powerful framework grounded in the theory of
reproducing kernel Hilbert spaces (RKHSs)~\cite{Aronszajn50,Hein04,RKHSreview,Paulsen16}; 
see Section~\ref{sec:framework} for the requisite mathematical preliminaries.
A key advantage of kernel-based GSA is its ability to treat models with
scalar-, vector-, or function-valued inputs and outputs under one general
framework. Also, these methods enable a natural treatment of models with
correlated inputs. 

Two sensitivity measures have proven particularly useful in the context of
kernel-based GSA: the {\em maximum mean discrepancy}~\cite{Barr22, daVeiga15},
and the {\em Hilbert--Schmidt Independence Criterion}
(HSIC)~\cite{daVeiga15,Greenfeld20, Gretton05b}.  Our work focuses on the HSIC,
a moment-independent statistic that quantifies the distance between the joint
law of the inputs and output and the product of the marginal laws. Under
appropriate choices of kernels, the HSIC between two random variables
vanishes if and only if they are independent~\cite{Gretton05b}. Consequently,
the condition $\HSIC(X_i,Y)=0$ indicates a lack of influence of $X_i$ on the
model output $Y$ in~\eqref{eq:model_basic}. Initial implementations of the HSIC
in GSA have used both the raw statistic and normalized variants to rank input
parameters based on their overall contribution to the uncertainty in the model
output~\cite{daVeiga21}.

A key barrier remains in using HSIC-based indices to construct interpretable sensitivity measures: in general, they fail to satisfy the crucial property of
{\em monotonicity under marginalization}. Specifically, we desire that all index subsets $A,B\subseteq\{1,\ldots,p\}$ satisfy\begin{equation} \HSIC(X_A,Y)\le \HSIC(X_B,Y) \quad \text{whenever} \ A\subseteq B.
\end{equation}
Attaining this property guarantees that the values of subsequently-defined sensitivity indices are increasing with respect to the number of inputs considered, thereby providing an interpretable ranking of parameter influence.

As detailed in Section~\ref{sec:monotonicity}, monotonicity under
marginalization can be ensured through an appropriate class of kernel functions. Namely,
by utilizing suitably augmented kernels, we can prove the requisite monotonicity
property. The analysis in Section~\ref{sec:monotonicity} also provides further
theoretical insight regarding the structure of the feature spaces associated with
commonly used kernels.
The framework developed in Section~\ref{sec:monotonicity} enables defining the \emph{first-order} and \emph{total HSIC sensitivity indices}; see Section~\ref{sec:totalHSICSec}. This extends and justifies the analogous indices presented in~\cite{daVeiga21} 
to the case of models with dependent inputs. The first-order and total
HSIC sensitivity indices are bounded, interpretable, and provide a moment-independent analogue
of the first-order and total-effect Sobol' indices, respectively. Collectively, these new HSIC-based sensitivity indices unify the
generality of kernel methods with the clear interpretability of Sobol'
indices.  These indices are also tractable to approximate numerically via sampling. 

The key contributions of this article are:
\begin{itemize}
	\item a rigorous analysis of the structural properties
	required to ensure the monotonicity of the HSIC under marginalization (see
	Section~\ref{sec:monotonicity});
	
	\item defining the first-order and total HSIC sensitivity indices (see Section~\ref{sec:totalHSIC}), which
	establishes two moment-independent sensitivity measures that are stable under marginalization 
	and accommodate models with correlated parameters; 
	
	
	\item a comprehensive set of computational experiments (see Section
	\ref{sec:numerics}) demonstrating the effectiveness of the first-order and
	total HSIC sensitivity indices across classical GSA benchmarks, including
	models with nonlinear relationships (Section~\ref{sec:ishigami}), correlated
	inputs (Section~\ref{sec:corrport}), and functional outputs
	(Section~\ref{sec:cholera}).
	\end{itemize} 
    %
    We also consider the computation of the HSIC indices 
    using a well-known empirical estimator for the
	HSIC; see Section~\ref{sec:theory}. 
    
    We point out that the primary focus of this article is the rigorous analysis 
    of the proposed HSIC indices. The numerical estimation of 
    HSIC indices for models with high-dimensional parameters is a subject 
    of future work. This is discussed further in 
    Section~\ref{sec:conc}, where we provide some concluding remarks.


\section{Mathematical Preliminaries}
\label{sec:framework}
While the theory of reproducing kernel Hilbert spaces (RKHSs) and the
Hilbert--Schmidt Independence Criterion (HSIC) can be formulated in far greater
generality, we restrict attention here to random variables taking values in real
separable Hilbert spaces.  This provides a convenient functional-analytic
setting, and applies to the common problems where inputs are real- or
vector-valued parameters and outputs may be time series or spatially distributed
field quantities.  Our goal in this section is to discuss the key elements of
RKHS theory essential for sensitivity analysis: kernel mean embeddings,
characteristic kernels, and the HSIC.

\subsection{Reproducing kernel Hilbert spaces}\label{sec:RKHS}
Let $D$ be a separable Hilbert space equipped with its Borel
$\sigma$-algebra $\cB(D)$.  A function $K: D\times D \to \bR$ is referred to as
a {\em kernel} if it is symmetric and positive definite. Kernels
inform the definition of an RKHS:
\begin{definition}\label{def:RKHS}
Let $\cH$ be a Hilbert space whose elements are real-valued functions on $D$. We say that $\cH$ is an RKHS if there is a kernel $k\in \cH$ so that for every $x\in D$, we have: \begin{enumerate} 
	\item {\em\bf{Inclusion:}} $k(x, \cdot) \in \cH$,
	\item {\em\bf{Reproduction}:} $f(x) = \langle f, k(x,\cdot) \rangle_\cH$ for every $f\in \cH$.
\end{enumerate}
\end{definition} 
Such a kernel $k$ is referred to as the (unique) {\em reproducing kernel} for
$\cH$. The RKHS formalism aligns with the core concept in machine learning of
transforming data into a more interpretable representation. Indeed, an RKHS
$\cH$ is often termed a {\em feature space}, and its reproducing kernel $k$
defines the {\em canonical feature map} $\Phi: D\to \cH$, where
$\Phi(x)=k(x,\cdot)$ for every $x\in D$. Exploiting the reproduction property of
$k$ yields the {\em kernel trick}: 
\begin{equation}\label{eq:kerneltrick} 
    \langle \Phi(x), \Phi(y)\rangle_\cH = k(x,y), \quad x,y\in D.
\end{equation} 
This relation states that evaluating the (often nonlinear) kernel $k$
corresponds to evaluating the (linear) inner product of the {\em features}
$\Phi(x)$ and $\Phi(y)$ in $\cH$. 

The Moore--Aronszajn theorem states that every
kernel $k$ uniquely determines a feature space $\cH_k$ on which the kernel trick
applies to $k$ \cite{Aronszajn50}.
\begin{theorem}\label{thm:moore}
	For every kernel $k: D\times D\to \bR$, there is a unique Hilbert space $\cH_k$ of real-valued functions on $D$ so that $\cH_k$ is an RKHS with $k$ as its reproducing kernel.
\end{theorem} 
As a consequence of Theorem~\ref{thm:moore}, we often refer to $\cH_k$ as the RKHS {\em
associated with} the kernel $k$. Subsequently, we relate the product of
kernels with the tensor product of their respective RKHSs. Recall that, for
separable Hilbert spaces $\cF$ and $\cG$, the {\em tensor product operator}
$f\otimes g: \cG\to \cF$ is defined as $(f\otimes g)(h) = \langle
g,h\rangle_\cG \, f$ for every $h\in \cG$.  Then the {\em tensor product
space} $\cF\otimes \cG \defeq \{f\otimes g: f\in \cF, g\in \cG\}$ equipped with the
inner product 
\begin{equation}\label{eq:tensorproductip} \langle f_1\otimes g_1,
f_2\otimes g_2\rangle_{\cF\otimes \cG} \defeq \langle f_1, f_2\rangle_\cF \langle
g_1,g_2\rangle_\cG, \quad f_1,f_2\in \cF, g_1,g_2\in \cG,
\end{equation} 
is an inner product space that is not necessarily complete. We denote by
$\cF\widehat{\otimes}\cG$ the Hilbert space obtained by taking the completion of
$\cF\otimes \cG$ with respect to the inner product in~\eqref{eq:tensorproductip}. 

The following result, which will be needed in the sequel, describes the RKHS
associated with the product of kernels. 
\begin{lemma}\label{lem:productkernel}
	For $i \in \{1,\ldots, p\}$, 
    let $\cX_i$ be a separable Hilbert space, $k_i: \cX_i\times \cX_i\to \bR$ a kernel, and $\cH_i$ the RKHS associated with $k_i$. 
    Define a map $k: \cX\times \cX\to \bR$ on the Cartesian product space $\cX = \bigtimes_{i=1}^p \cX_i$  by 
    \begin{equation} 
        k((x_1, \dots, x_p),(y_1,\dots, y_p)) := \prod_{i=1}^p k_i(x_i,y_i).
    \end{equation} 
    Then $k$ is a kernel and $\cH = \widehat{\otimes}_{i=1}^p \cH_i$ is the RKHS associated with $k$.
\end{lemma}
Henceforth, we refer to a kernel of the form Lemma~\ref{lem:productkernel} as a {\em
product kernel}.  
In the context of uncertainty quantification, this result allows us to describe
the RKHS associated with kernels defined on the input space for models of the
form $Y=f(X)$, where $X=(X_1, \dots, X_p)$. Analysis of the output $Y$ requires
either knowing or estimating the distributions for each $X_i$.
Therefore, a discussion of incorporating uncertainty in the RKHS framework is
crucial before presenting our problem setup. This is done in the next subsection.

\subsection{RKHS embeddings of probability measures}\label{sec:char} 
Let $D$ be a separable Hilbert space and consider the measurable space 
$(D, \cB(D))$. Let $k: D\times D\to \bR$ be a kernel with canonical feature map $\Phi$ and
associated RKHS $\cH$. 
%
%
%
We denote by $\mathcal{M}(D)$ the set of all probability
measures on $(D,\cB(D))$. Within $\mathcal{M}(D)$, the subset $\mathcal{M}_k(D)\defeq  \{\mathbb{P} \in \mathcal{M}(D) \ \vert \ \int_D \|\Phi(x)\|_{\cH} \bP(dx)<\infty  \}$ consists of all probability measures on $(D,\cB(D))$ for which the canonical feature map $\Phi$ is Bochner integrable. It should be noted that $\mathcal{M}(D) = \mathcal{M}_k(D)$ whenever the kernel $k$ is bounded. We define the {\em kernel mean embedding} (KME) with
respect to $k$ as the map $\eta_k: \mathcal{M}_k(D) \to \cH$ given by  
\begin{equation}
    \eta_k(\bP) := \int_D k(x,\cdot) \bP(dx), \quad \bP\in \mathcal{M}_k(D).
\end{equation} 
In general, the KME does not uniquely characterize elements of $\mathcal{M}_k(D)$.
It does, however, in the case where the kernel $k$ is characteristic, as defined 
next:
\begin{definition}\label{def:char} 
Let $D$ be a separable
Hilbert space and $k:D\times D\to \bR$ be a kernel with associated RKHS $\cH$.
We say that $k$ is {\em characteristic} if the KME $\eta_k: \mathcal{M}_k(D)\to
\cH$ is injective. 
\end{definition} 

It is known that the {\em Gaussian kernel} with {\em bandwidth} $\sigma\in \bR\setminus\{0\}$
given by \begin{equation} k_\sigma(x,x') := \exp\left(-\|x-x'\|_D^2/ 2\sigma^2
\right), \quad x,x'\in D,\end{equation} is characteristic when $D\subseteq
\bR^n$ \cite{Micchelli06, Sriperumbudur11, Szabo18}. Research on when kernels
defined on non-compact spaces are characteristic is at present relatively
scarce; however, Ziegel et al.~\cite[Theorem 3.1]{Ziegel24} showed that the Gaussian kernel
is characteristic on every separable Hilbert space. As such, we rely on Gaussian
kernels to inform our setup for kernel-based sensitivity analysis. 

The following result, which is needed in what follows, demonstrates that the product of Gaussian kernels defined on separable Hilbert spaces is characteristic on the product space.
\begin{lemma}\label{lem:productchar} For $i=1,\dots, p$, let $\cX_i$ be a separable Hilbert space, let $k_{\sigma_i}: \cX_i\times
\cX_i\to \bR$ be a Gaussian kernel with bandwidth parameter $\sigma_i$, and let $k := \prod_{i=1}^p k_{\sigma_i}$ be the
product kernel defined on $\cX=\bigtimes_{i=1}^p \cX_i$ as in
Lemma~\ref{lem:productkernel}. Then $k$ is a characteristic kernel. 
\end{lemma} 

\begin{proof}
For each space $(\cX_i, \|\cdot\|_{\cX_i})$, define the weighted norm $\|\cdot\|_{\cX_i, w_i}$ as follows:
\begin{equation}
	\|x_i\|_{\cX_i, w_i}\defeq \frac{1}{\sqrt{2} \sigma_i} \|x_i\|_{\cX_i}, \quad x_i\in \cX_i.
\end{equation} Since $\sigma_i \neq 0$, the norms $ \|\cdot\|_{\cX_i}$ and $\|\cdot\|_{\cX_i, w_i}$ are equivalent on $\cX_i$. As such, $(\cX_i, \|\cdot\|_{\cX_i, w_i})$ is a separable Hilbert space. It follows that the product space $(\cX, \|\cdot\|_{\cX, w})$, where 
\begin{equation}
	\|x\|_{\cX, w}^2 = \sum_{i=1}^p \|x_i\|_{\cX_i, w_i}^2, \quad x\in \cX,
\end{equation} is also a separable Hilbert space; see, e.g.,~\cite[Ex. 10, pg. 195]{munkres2000}. Then for any $x,x'\in \cX$, 
\begin{equation}\label{eq:gaussprod}
	k(x,x') = \prod_{i=1}^p k_{\sigma_i}(x_i,x_i') = \exp\left(-\sum_{i=1}^p \|x_i-x_i'\|_{\cX_i, w_i}^2\right) = \exp(-\|x-x'\|_{\cX, w}^2).
\end{equation} As such,~\eqref{eq:gaussprod} shows that $k$ is a Gaussian kernel with bandwidth $\sigma = 1/\sqrt{2}$ defined on a separable Hilbert space and is therefore characteristic (per~\cite[Theorem 3.1]{Ziegel24}).
\end{proof}

\subsection{Problem setup and the HSIC}\label{rem:setup}
For the remainder of this article, we assume the following setup.  We consider a
mathematical model $$Y = f(X), \quad X = (X_1,\dots, X_p),$$ as a measurable map
on a probability space $(\Omega, \cF, \bP)$. 
For each $i \in \{1, \ldots, p\}$, the input $X_i$ is a random variable, $X_i:
(\Omega, \cF, \bP) \to (\cX_i, \cB(\cX_i))$, where $\cX_i$ is a separable
Hilbert space and $\cB(\cX_i)$ is the corresponding Borel $\sigma$-algebra.
Similarly, the output is a random variable $Y:(\Omega, \cF, \bP) \to (\cY,
\cB(\cY))$, where $\cY$ is a separable Hilbert space equipped with the Borel $\sigma$-algebra $\cB(\cY)$. We denote the respective laws of the inputs and the output by 
$\bP_{X_i}$ and $\bP_Y$. 

For an index subset $A\subseteq
\{1,\dots, p\}$, we define $X_A = (X_i)_{i\in A}$ as a random variable that
takes values in the product space $\cX_A = \bigtimes_{i\in A} \cX_i$ equipped with
the product $\sigma$-algebra $\cB(\cX_A)$ and product law $\bP_{X_A}$. For ease
of notation, we omit the brackets from all relevant subscripts when $A = \{i\}$;
similarly, we remove the subscripts entirely when $A = \{1,\dots, p\}$. Finally,
we equip $\cX_A$ with a product kernel $k_A:=\prod_{i\in A}
k_i$ and $\cY$ with a kernel $\ell$, assuming also that $\bP_{X_A}\in \mathcal{M}_{k_A}(\cX_i)$ and $\bP_{Y}\in \mathcal{M}_\ell(\cY)$. We denote the associated
RKHSs by $\cH_A$ and $\cG$, respectively. A depiction of the problem setup can be 
found in Figure~\ref{fig:diagram}.

This setup provides the probabilistic foundation for kernel-based sensitivity
analysis.  
A core idea in this context is to construct a cross-covariance
operator $C_{X_A Y}: \cG \to \cH_A$ given by 
\begin{equation}\label{eq:cross_covariance_operator}
    \begin{aligned} 
        C_{X_AY} &\defeq \int_{\cX_A\times \cY} (\Phi_{k_A}(x) - M_{\cX_A}) \otimes (\Phi_\ell(y) - M_\cY)\ \bP_{X_A,Y}(dx,dy)\\ 
        &= \int_{\cX_A \times \cY} \Phi_{\cX_A}(x) \otimes \Phi_\cY(y) \ \bP_{X_A,Y}(dx,dy) - M_{\cX_A} \otimes M_{\cY},
    \end{aligned}
\end{equation} 
where $\Phi_{X_A}$ and $\Phi_Y$ are the canonical feature maps of $k_A$ and $\ell$, respectively, and $\bP_{X_A,Y}$ denotes the joint law of $(X_A,Y)$. Under the assumptions of Section \ref{rem:setup}, it is straightforward to show that $C_{X_AY}$ is a Hilbert--Schmidt operator \cite{daPrato06}.

\begin{figure}[htbp]
\centering
\includegraphics[width=0.9\textwidth]{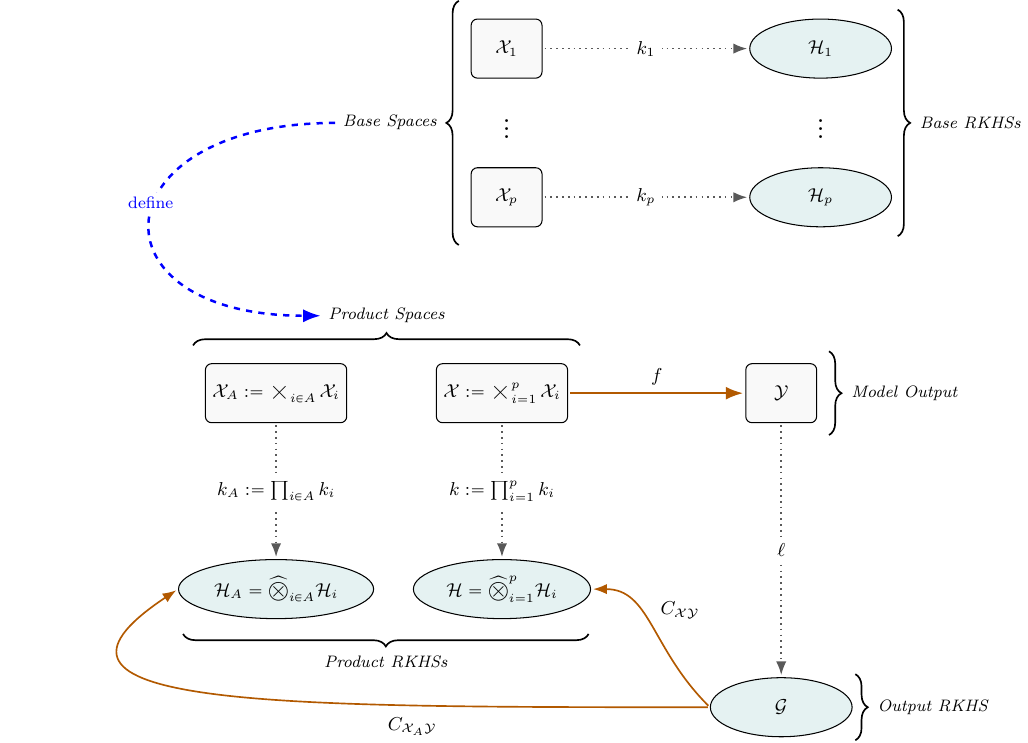}
\caption{A depiction of the problem setup.} 
\label{fig:diagram}
\end{figure}

 \begin{definition}\label{def:HSIC} The Hilbert--Schmidt independence criterion (HSIC) between $X_A$ and $Y$ is denoted by $\HSIC(X_A,Y)$ and defined by the squared Hilbert--Schmidt norm of $C_{X_AY}$; that is, \begin{equation}\HSIC(X_A,Y) := \|C_{X_AY}\|_{HS}^2.\end{equation}\end{definition} 

The authors of \cite{Gretton05a} use the kernel trick to show that the HSIC
between $X_A$ and $Y$ can be completely characterized by the kernels endowed
upon the spaces $\cX_A$ and $\cY$. 
This establishes the HSIC as a
kernel-based tool to characterized independence between $X_A$ and $Y$, 
as summarized in the following result: 
\begin{theorem}\label{thm:HSICind} 
Let $k_A$ and $\ell$ be characteristic kernels on $\cX_A$ and $\cY$,
respectively. Then,  
$$\HSIC(X_A,Y)=0 \quad \text{if and only if} \quad \bP_{X_A,Y} = \bP_{X_A}\otimes \bP_Y.$$
\end{theorem} 


When using characteristic kernels, the cross-covariance operator
fully encodes the joint distribution, and its vanishing indicates that the joint
law factorizes. It follows from Lemma~\ref{lem:productchar} that equipping separable
Hilbert spaces with Gaussian kernels allows HSIC to generalize the classical
idea that zero covariance implies independence. In the following section, we
will show that such a setup is not enough to guarantee the monotonicity of the
HSIC under marginalization.


\section{Monotonicity of the HSIC under marginalization}\label{sec:monotonicity}
The discussion in the previous section motivates the use of HSIC as a
sensitivity measure. Namely, one may consider the \emph{raw HSIC score}
$\HSIC(X_A,Y)$ as an indication of the influence of $X_A$ upon $Y$. 
While appealing in its simplicity, the raw HSIC is difficult to interpret.  In
particular, a large raw HSIC score has no direct interpretation as “influence”
of an input. 
To address this limitation, normalized variants have been 
proposed. The prevalent approach, suggested in~\cite{daVeiga15}, 
is to use the {\em distance correlation index} between $X_A$ and $Y$ given by
\begin{equation}\label{eq:dcorr}
    R_{X_A}^2(Y) \defeq \frac{\HSIC(X_A,Y)}{\sqrt{\HSIC(X_A,X_A)\HSIC(Y,Y)}}.
\end{equation} 
This rescaling provides a correlation-like interpretation; see~\cite{daVeiga15}
for details. 
Despite their utility, 
these
indices suffer from a lack of a ranking ``yardstick" due to the fact that the
normalization in \eqref{eq:dcorr} depends on $X_A$. 
Furthermore, the definition of the distance correlation indices does not ensure
that $R_{X_i}^2(Y) \le R_{X_{\{i,j\}}}^2(Y)$; that is, the normalization masks
how dependence behaves under marginalization.

Without monotonicity under marginalization, we cannot construct analogues of
Sobol' indices, nor can we interpret HSIC-based indices as
fractions of a quantity explaining the contributions of all the inputs
to the model output.
In what follows, we show how a suitable kernel construction ensures
monotonicity. We require that, for $A, B \subseteq \{1, \ldots, p\}$,
\[
\HSIC(X_A,Y)\le \HSIC(X_B,Y), \quad \text{whenever} \quad A \subseteq B.
\]
When satisfied, this property lays the foundation for
defining the first-order and total HSIC sensitivity indices (see Section~\ref{sec:totalHSIC}),
which are bounded between 0 and 1 and provide a ranking mechanism to compare the
relative importance of parameter subsets.

\subsection{Augmented kernel construction}
As it will become apparent in the discussion that follows, a critical component
of our analysis is that the constant function $\1$ must belong to the RKHS
generated from $\cX_i$ for all $i\in \{1,\dots, p\}$. We also require that the
kernels $k_i$ defined on $\cX_i$ are characteristic so that 
Theorem~\ref{thm:HSICind} holds.

It is natural to wonder if Gaussian kernels
ensure monotonicity under marginalization; however, it is well-established that
the RKHS associated with a Gaussian kernel does not contain the constant
function $\1$ \cite{Steinwert06}.  
We provide a two-part remedy for this issue. 
With Gaussian kernels in mind as a starting point, we construct a class of {\em
augmented kernels} that are characteristic on every separable Hilbert space and
generate RKHSs that contain the constant function $\1$. We then use the
augmented kernels to state and prove Theorem~\ref{thm:mono}, the core result of
this article, which guarantees monotonicity of the HSIC under marginalization. 

The following result from \cite{Durrande11} informs our kernel construction. Although the cited reference states the result for $\cX\subset \bR$, the proof extends verbatim to arbitrary separable Hilbert spaces.

\begin{lemma}[\cite{Durrande11}, Proposition 1]\label{lem:RKHSdecomp}
	Let $\cX$ be a separable Hilbert space, and equip the probability space $(\cX, \cB(\cX), \bP_X)$ with the kernel $k$. Assume that $\bP_X\in \mathcal{M}_k(\cX)$, and denote the associated RKHS by $\cH$. Then $\cH$ admits the orthogonal decomposition 
	\begin{equation}\label{eq:RKHSdecomp}  
	\cH = \cH^1 \oplus \cH^0, 
	\end{equation} where $\cH^0$ is an RKHS of zero-mean functions with respect to $\bP_X$, and $\cH^1$ is an RKHS with dimension of at most one.
\end{lemma} 


In what follows, we first
describe how to transform any given kernel $k$ into one whose associated RKHS consists
only of mean-zero functions with respect to $\bP_X$. We then show that this
space $\cH^0$ does not contain constant functions and demonstrate the
orthogonality described in \eqref{eq:RKHSdecomp} between $\cH^0$ and
$\text{span}\{\1\}$. Finally, we define an augmented kernel as one whose
associated RKHS admits the desired decomposition $\cH = \text{span} \{\1\} \oplus
\cH^0$. We will show that augmented kernels are a correct choice in guaranteeing
monotonicity of the HSIC under marginalization.

\begin{definition}\label{def:centeredker} 
Let $(\cX, \cB(\cX), \bP_X)$ be a probability space and $k:\cX\times \cX\to \bR$ be a kernel. The {\em centering} of $k$ with respect to $\bP_X$ is given by $k_c = k_c(k,\bP_X): \cX\times \cX\to \bR$, where \begin{equation}\label{eq:center} k_c(x,x') \defeq k(x,x') - \int_{\cX} k(x,s) \bP_X(ds) - \int_{\cX} k(x',s) \bP_X(ds) + \int_{\cX}\int_{\cX} k(s,t) \bP_X(ds)\bP_X(dt).\end{equation}
\end{definition}

First, we show that $k_c$ is a kernel on $\cX$. 

\begin{lemma}\label{lem:kcisker}
	Let $(\cX, \cB(\cX), \bP_X)$ be a probability space and $k:\cX\times \cX\to \bR$ be a kernel. If $\bP_X\in \mathcal{M}_k(\cX)$, then the centering of $k$ with respect to $\bP_X$ is a kernel on $\cX$.
\end{lemma}

\begin{proof}
It is immediate from \eqref{eq:center} that $k_c$ is symmetric on $\cX$.  It
therefore suffices to show that $k_c$ is positive definite.  
This becomes clear once we express $\Phi_c$, the canonical feature map of $k_c$,
in terms of $\Phi_k$, the canonical feature map of $k$.
Since $\bP_X\in \mathcal{M}_k(\cX)$, we know that
$\eta_k(\bP_X)$ is an element of the RKHS $\cH$ associated with $k$. Using
linearity of the inner product and Bochner integral, it follows that for every
$x,x'\in \cX$, 
\begin{align*}
	k_c(x,x') &= \langle \Phi_k(x), \Phi_k(x')\rangle_\cH - \langle \Phi_k(x), \eta_k(\bP_X)\rangle_\cH - \langle \Phi_k(x;), \eta_k(\bP_X)\rangle_\cH + \|\eta_k(\bP_X)\|_{\cH}^2\\
	&= \langle \Phi_k(x) - \eta_k(\bP_X), \Phi_k(x') - \eta_k(\bP_X)\rangle_\cH \\
	&= \langle \Phi_{c}(x), \Phi_c(x') \rangle_\cH, 
\end{align*} 
where 
$\Phi_{c}(x) \defeq \Phi_k(x) - \eta_k(\bP_X)$ for every $x\in \cX$.
The positive definiteness of $k_c$ follows from this relation.
\end{proof}

The following result shows that the canonical feature map associated with $k_c$
is mean-zero with respect to $\bP_X$.

\begin{lemma}\label{lem:KMEcenter} 
Let $(\cX, \cB(\cX), \bP_X)$ be a probability space and $k:\cX\times \cX\to \bR$ be a kernel. Assume that $\bP_X\in \mathcal{M}_k(\cX)$. Let $k_c$ be the centering of $k$ with respect to $\bP_X$ and denote by $\Phi_c$ the canonical feature map associated with $k_c$. Then $\int_{\cX} \Phi_c(x) \bP_X(dx) = 0$.
\end{lemma} 
\begin{proof} 
We compute directly that \begin{equation}\label{eq:expand} \begin{split}
\int_{\cX} \Phi_c(x) &\bP_X(dx) = \int_{\cX} k(x,\cdot) \bP_X(dx) - \int_{\cX}\int_{\cX} k(x,s) \bP_X(ds)\bP_X(dx)\\
& - \int_{\cX}\int_{\cX} k(\cdot,s) \bP_X(ds)\bP_X(dx) + \int_{\cX}\int_{\cX}\int_{\cX} k(s,t) \bP_X(ds)\bP_X(dt)\bP_X(dx)
\end{split}\end{equation} Notice that the integrands of the third and fourth terms in the right-hand side of~\eqref{eq:expand} are independent of $x$. We therefore simplify~\eqref{eq:expand} to read as follows: \begin{equation}\label{eq:expand2}\begin{split}
\int_{\cX} \Phi_c(x)\bP_X(dx) = \int_{\cX} &k(x,\cdot) \bP_X(dx) - \int_{\cX}\int_{\cX} k(x,s) \bP_X(ds)\bP_X(dx)\\
& - \int_{\cX} k(\cdot,s) \bP_X(ds) + \int_{\cX}\int_{\cX} k(s,t) \bP_X(ds)\bP_X(dt).
\end{split}\end{equation} The second and fourth terms in the right-hand side of~\eqref{eq:expand2} are equal due to the symmetry of the kernel function $k$. The first and third terms in this equation are equal by the same property. It follows that \begin{equation}\label{eq:expand3} \int_{\cX} \Phi_c(x)\bP_X(dx) = 0,\end{equation} which is the desired conclusion.
\end{proof}

With this result in mind, we can use the reproduction property of $k_c$ to show that its associated RKHS $\cH^0$ consists only of mean-zero functions with respect to $\bP_X$. Indeed, for every $f\in \cH^0$, we have 
\begin{equation}
\int_\cX f(x) \bP_X(dx) = \int_\cX \langle f, \Phi_c(x)\rangle_{\cH^0}\ \bP_X(dx) = \langle f, \int_\cX \Phi_c(x) \bP_X(dx)\rangle_{\cH^0} = \langle f, 0\rangle_{\cH^0} = 0,
\end{equation} 
where the penultimate equality follows due to Lemma~\ref{lem:KMEcenter}. This
characterization of functions in $\cH^0$ allows us to describe the space $\cH^0$
as the image of $\cH_k$, the RKHS associated with the original kernel $k$, under
a linear operator.  Specifically, $\cH^0 = P(\cH_k)$, where
\begin{equation}\label{eq:proj}
	Pf \defeq f - \int_\cX f(x) \ \bP_X(dx), \quad f \in \cH_k.
\end{equation} 
It is straightforward to note that $P^2=P$, which yields the interpretation that
the centering of $k$ with respect to $\bP_X$ is a projection of $\cH_k$ onto the
subspace of mean-zero functions with respect to $\bP_X$. We are particularly
interested in the setting where $P$ is injective. Note that
\begin{equation}\label{eq:kerP}
\text{ker}P = \{f\in \cH_k \ \vert \ f = \int_{\cX} f(x) \bP_X(dx)\}.
\end{equation} 
It follows that $\text{ker}P = \text{span}\{\1\}$ if $\1\in \cH_k$ and  $\text{ker}P =
\{0\}$ otherwise. In short, the centering of $k$ with respect to $\bP_X$ is only bijective when $\1\notin \cH_k$. 
As such, it is advantageous to consider the centering of Gaussian kernels since
their associated RKHSs do not contain $\1$.

Finally, we describe how to construct an RKHS $\cH = \text{span}\{\1\} \oplus
\cH^0$ as in Lemma~\ref{eq:RKHSdecomp}. While the ``base kernel" $k$ may
be any kernel whose associated RKHS does not contain the constant
function $\1$, it is advantageous to use the Gaussian kernels for convenience.  

\begin{definition}\label{def:aug}
	Let $(\cX, \cB(\cX), \bP_X)$ be a probability space equipped with kernel $k$, and assume that $\bP_X\in \mathcal{M}_k(\cX)$. Denote by $\cH_k$ the RKHS associated with $k$, and assume that $\1\notin \cH_k$. Then the {\em augmented kernel} $k^* = k^*(k, \bP_X)$ defined by \begin{equation}
	k^*(x,x') \defeq k_c(x,x') + 1, \quad x,x'\in \cX,
	\end{equation} is a kernel whose associated RKHS $\cH$ admits the orthogonal decomposition $\cH = \text{span}\{\1\} \oplus
\cH^0$, where $\cH^0$ consists of mean-zero functions with respect to $\bP_X$.
\end{definition}

\paragraph{Remark} Other authors refer to these kernels as (one-dimensional) {\em ANOVA kernels} due to their role in defining a functional ANOVA decomposition under the assumption of independent inputs \cite{daVeiga21, Durrande11}. Our choice in nomenclature
is made to highlight the lack of assumptions we impose on the dependence of
inputs. The following result proves that the product of augmented kernels with base Gaussian kernels is characteristic.

\begin{theorem}\label{thm:prodchar}
	Let $\{(\cX_i, \cB(\cX_i), \bP_i)\}_{i=1}^p$ be a collection of 
    probability spaces, where the $\cX_i$ are all separable Hilbert spaces. Let $k_i: \cX_i\times \cX_i\to \bR$, $i \in \{1, \ldots, p\}$, be a 
    collection of Gaussian kernels with respective bandwidth parameters $\{\sigma_i\}_{i=1}^p$. Denote by $k^*_i = k^*_i(k_i, \bP_i)$ the augmented Gaussian kernel on each space. The product kernel $k^* \defeq \prod_{i=1}^p k^*_i$ is characteristic on $\cX\defeq \bigtimes_{i=1}^p \cX_i$. 
\end{theorem}

\begin{proof}
	Recall that $k_i$ is bounded for each $i$, so $k_i^*$ is bounded for each
	$i$. It follows that $\mathcal{M}_{k_i}(\cX_i) = \mathcal{M}(\cX_i)  
    = \mathcal{M}_{k_i^*}(\cX_i)$.
    Moreover, letting $k\defeq \prod_{i=1}^p k_i$, 
    allows us to write that 
	$\mathcal{M}_{k}(\cX) = \mathcal{M}(\cX) = \mathcal{M}_{k^*}(\cX)$. Let $\mathbb{Q}_1, \mathbb{Q}_2\in \mathcal{M}(\cX)$, and
	assume that $\eta_{k^*}(\mathbb{Q}_1) = \eta_{k^*}(\mathbb{Q}_2)$. 
	Now, we will show that $\mathbb{Q}_1=\mathbb{Q}_2$. 
    
    Denote by $\cH =
	\widehat{\otimes}_{i=1}^p \cH_i$ the RKHS associated with $k^*$. From
	$\eta_{k^*}(\mathbb{Q}_1) = \eta_{k^*}(\mathbb{Q}_2)$ we have 
	\begin{equation}\label{eq:step1} \int_\cX \langle f,
	k^*(x,\cdot)\rangle_{\cH} \mathbb{Q}_1(dx) = \int_\cX \langle f,
	k^*(x,\cdot)\rangle_{\cH} \mathbb{Q}_2(dx), \quad \text{for all } f\in \cH.  
    \end{equation}
	Using the definition of the inner product on $\cH$,~\eqref{eq:step1} becomes
	\begin{equation}\label{eq:step2} 
        \int_\cX \prod_{i=1}^p \langle f_i,
	    k_i^*(x_i,\cdot)\rangle_{\cH_i} \mathbb{Q}_1(dx) = \int_\cX \prod_{i=1}^p
	    \langle f_i, k_i^*(x_i,\cdot)\rangle_{\cH_i} \mathbb{Q}_2(dx), \quad 
        \text{for all } f_i\in\cH_i.  
    \end{equation} 
    Let $t_i \in \cX_i$, $i \in \{1, \dots, p\}$, 
    be fixed but arbitrary, and set $f_i\defeq k_i(t_i, \cdot)\in \cH_i$ for each $i$.
	Then, 
    \begin{equation}\label{eq:fi} f_i = k_{i,c}(t_i,\cdot)
	+\mu_i(t_i)\cdot\1 + \mu_i(\cdot) - b_i, 
    \end{equation} 
    where 
	$\mu_i(\cdot) \defeq \int_{\cX_i} k_i(\cdot,s) \bP_i(ds)$, $b_i\defeq
	\int_{\cX_i} \int_{\cX_i} k_i(s,s') \bP_i(ds)\bP_i(ds')$, 
	and $k_{i,c}$ denotes the centering of $k_i$ with respect to
	$\bP_i$. It is straightforward to show that $\mu_i(\cdot)-b_i$ is mean-zero
	with respect to $\bP_i$. Then, using 
    the orthogonality of mean-zero (with respect to $\bP_i$) 
    functions and $\text{span}\{ \1 \}$ within $\cH_i$   
    (cf.~Lemma~\ref{lem:RKHSdecomp}), we have 
    \begin{equation}\label{eq:key}
    \begin{aligned} \langle f_i,k_i^*(x_i,\cdot)\rangle_{\cH_i} &= \langle k_{i,c}(t_i,\cdot)
	+\mu_i(t_i)\cdot \1 + \mu_i(\cdot) - b_i\cdot \1, k_{i,c}(x_i,\cdot) +
	\1(x_i)\rangle_{\cH_i}\\ 
        &= \langle k_{i,c}(t_i,\cdot) + \mu_i(\cdot) -
	b_i\cdot \1, k_{i,c}(x_i,\cdot)\rangle_{\cH_i} + \mu_i(t_i)\langle \1,
	\1(x_i)\rangle_{\cH_i}\\ 
        &= k_{i,c}(t_i,x_i) + \mu_i(x_i) - b_i\cdot \1(x_i)
	    + \mu_i(t_i)\cdot\1(x_i)\\
    &= k_i(t_i,x_i).
    \end{aligned} 
    \end{equation}
    Note that 
    the penultimate equality follows from the reproducing property of the
	kernels, and the final equality follows from~\eqref{eq:fi}.
	Substituting~\eqref{eq:key} into~\eqref{eq:step2} yields that
	\begin{equation}\label{eq:step3} \int_\cX \prod_{i=1}^p k_i(t_i,x_i)
	\mathbb{Q}_1(dx) = \int_\cX \prod_{i=1}^p k_i(t_i,x_i) \mathbb{Q}_2(dx).
	\end{equation} Since the $t_i$'s are arbitrary, we conclude
	from~\eqref{eq:step3} that 
    \begin{equation}\label{eq:step4}
	 \eta_k(\mathbb{Q}_1) 
        = \int_\cX k(x, \cdot) \mathbb{Q}_1(dx)
	= \int_\cX k(x,\cdot) \mathbb{Q}_2(dx) =
	\eta_k(\mathbb{Q}_2).  
    \end{equation} 
    Finally, invoking Lemma~\ref{lem:productchar}, 
    which states that $k$ is characteristic on $\cX$, we have 
	$\mathbb{Q}_1 = \mathbb{Q}_2$. 
	\end{proof}

\paragraph{Remark} While we focus on augmented kernels with base Gaussian kernels in
Theorem~\ref{thm:prodchar}, this assumption can be relaxed. Indeed, any bounded
kernels $k_i$ whose associated RKHSs do not contain constant functions can be
employed as base kernels, so long as their product kernel $k$ is characteristic
on $\cX$. This condition is difficult to satisfy, as the product of
characteristic kernels is not always characteristic on the corresponding product
space~\cite{Szabo18}.  As such, we focus our attention on Gaussian base kernels.

\subsection{Ensuring monotonicity}
With the developments in the previous subsection in place, we are now ready to
state and prove Theorem~\ref{thm:mono}, which is the core result of this
article. This result shows that, by using products of augmented kernels whose base kernels are Gaussian, we can ensure the monotonicity of the HSIC under marginalization. 

\begin{theorem}\label{thm:mono}
Let $\{(\cX_i, \cB(\cX_i), \bP_i)\}_{i=1}^p$ be a collection of probability spaces, where the $\cX_i$ are all separable Hilbert spaces. Let $k_i^*: \cX_i\times \cX_i\to \bR$, $i\in\{1,\dots, p\}$ be a collection of augmented kernels whose base kernels are Gaussian with respective bandwidths $\{\sigma_i\}_{i=1}^p$. Similarly, let $(\cY, \cB(\cY), \bP_Y)$ be a probability space, where $\cY$ is a separable Hilbert space equipped with $k_Y^*$, an augmented kernel whose base kernel is Gaussian with bandwidth $\sigma_Y$. For any index subset $A\subseteq \{1,\dots, p\}$, endow the product space $\cX_A$ with the product kernel $k^*_A \defeq \prod_{i\in A} k_i^*$. Then for every $A\subseteq B\subseteq \{1,\dots, p\}$, we have 
\begin{equation}\HSIC(X_A,Y) \le \HSIC(X_B,Y).
\end{equation}
\end{theorem} 
\begin{proof}
Without loss of generality, we assume that $B=\{1,\dots,p\}$. Also, since 
the conclusion of the theorem holds trivially when $A = B$, we assume $A$ is a proper 
subset of $B$. We consider $\cH = \cH_A
\widehat\otimes \cH_{A^c}$ and define $\mathbf{1}_{A^c} \in \cH_{A^c}$ so that
$\mathbf{1}_{A^c} = \mathbf{1}\otimes \dots \otimes \mathbf{1}$ is the tensor
product of $p-|A|$ copies of $\mathbf{1}$.
Note that since the RKHSs associated with augmented kernels contain the constant function 
$\1$ (cf.~Lemma~\ref{lem:RKHSdecomp}), we know that $\mathbf{1}_{A^c}$ is well-defined. 
Define $i: \cH_A\to \cH$ so that
$i(f_A) = f_A\otimes \mathbf{1}_{A^c}$ for every $f_A\in \cH_A$. It is
straightforward to show that $i$ is a bounded linear operator with operator norm
equal to 1. The adjoint $i^*$ of $i$ is therefore a bounded linear operator with
the same norm. 

    
We now claim that $C_{X_AY} = i^*C_{XY}$. It suffices to show that for every
$f_A\in \cH_A$ and $g\in \cG$,%
\begin{equation}\label{eq:covequality}
        \langle f_A, C_{X_A Y}(g)\rangle_{\cH_A} = \langle f_A, i^*C_{XY}(g)\rangle_{\cH_A}.
\end{equation} 
Recalling~\eqref{eq:cross_covariance_operator}
and 
applying the kernel trick to the left-hand side of \cref{eq:covequality} yields 
\begin{equation}\label{eq:lhs} \begin{split}
    \langle f_A, C_{X_A Y}(g)\rangle_{\cH_A} &= \int_{\cX_A\times \cY} f_A(x_A)g(y)\bP_{X_A\times Y}(dx_A,dy) \\
        & \quad \quad - \int_{\cX_A} f_A(x_A)\bP_{X_A}(dx_A)\int_{\cY} g(y)\bP_{Y}(dy). \end{split}
\end{equation} 
Considering the right-hand side of $\cref{eq:covequality}$, we note 
\begin{equation}\label{eq:rhsip}\langle f_A,i^*C_{XY}(g)\rangle_{\cH_A}=\langle i(f_A),C_{XY}(g)\rangle_{\cH}.
\end{equation} 
After applying the kernel trick and Fubini's Theorem to \cref{eq:rhsip}, we obtain 
     \begin{equation}\begin{split}
         \langle i(f_A),C_{XY}(g)\rangle_{\cH}
         &= \int_{\cX_A\times \cY} f_A(x_A)g(y)\bP_{X_A\times Y}(dx_A,dy) \\
         &\quad \quad - \int_{\cX_A} f_A(x_A)\bP_{X_A}(dx_A)\int_{\cY} g(y)\bP_{Y}(dy), \end{split}
     \end{equation} which is precisely \cref{eq:lhs}. It follows that $C_{X_AY}=i^*C_{XY}$. Moreover, we have that \begin{equation}\label{eq:normineq} \|C_{X_A Y}\|_{HS} = \|i^*C_{XY}\|_{HS} \le \|C_{XY}\|_{HS}\end{equation} due to the submultiplicativity of the Hilbert--Schmidt norm under compositions with bounded linear operators. Squaring both sides of \cref{eq:normineq} yields that $\HSIC(X_A,Y)\le \HSIC(X,Y)$.
\end{proof}
\section{The first-order and total HSIC sensitivity indices}\label{sec:totalHSICSec}
In this section, we use the setting of Theorem~\ref{thm:mono} to define
the total HSIC sensitivity index. We also discuss efficient methods for
computing these indices. 

\subsection{Defining the sensitivity indices}\label{sec:totalHSIC}
Let $\{(\cX_i, \cB(\cX_i), \bP_i)\}_{i=1}^p$ be a collection of probability spaces, where the $\cX_i$ are all separable Hilbert spaces. Let $k_i^*: \cX_i\times \cX_i\to \bR$, $i\in\{1,\dots,p\}$ be a collection of augmented kernels whose base kernels are Gaussian with respective bandwidths $\{\sigma_i\}_{i=1}^p$. Similarly, let $(\cY, \cB(\cY), \bP_Y)$ be a probability space, where $\cY$ is a separable Hilbert space equipped with $k_Y^*$, an augmented kernel whose base kernel is Gaussian with bandwidth $\sigma_Y$. For any index subset $A\subseteq \{1,\dots, p\}$, endow the product space $\cX_A$ with the product kernel $k^*_A \defeq \prod_{i\in A} k_i^*$. We then define the \emph{first-order} and {\em total HSIC sensitivity indices} as 
\begin{equation}
\mathbb{F}_A(Y):= \frac{\HSIC(X_{A},Y)}{\HSIC(X,Y)}, \quad \mathbb{T}_A(Y):= 1 - \frac{\HSIC(X_{\sim A},Y)}{\HSIC(X,Y)},
\end{equation} 
where $X_{\sim A} = (X_i)_{i\notin A}$. For ease of notation, we write $\mathbb{F}_i(Y)$ and $\mathbb{T}_i(Y)$ whenever $A=\{i\}$. 

This definition mirrors the structure of the first-order and total-effect Sobol’
indices, but quantifies much more of the distributional dependence between $X_A$ and $Y$ than variance alone. We note that the collection of the total HSIC sensitivity indices does not form an
additive decomposition: the sum of $\mathbb{T}_A(Y)$ over all index subsets $A$ need not equal 1. This is not a defect, but a reflection of the fact that $\mathbb{T}_A(Y)$ captures all the effects of $X_A$ upon $Y$ in a manner analogous to total-effect Sobol’
indices. Since our setup does not require independence of the inputs, the first-order HSIC indices capture the first-order and correlation-based influence of a parameter subset upon the model output. Analogously, the total HSIC indices capture all first-order, joint, and correlation-based effects. Like the Sobol’ indices, the first-order and total HSIC sensitivity indices measure a ``share” of influence, but now in broader generality. 

By definition, we interpret $\mathbb{F}_A(Y)$ as the share of $\HSIC(X,Y)$ that can be attributed to the subset $X_A$. In addition to this explanation, we may also interpret the index $\mathbb{T}_A(Y)$ as the relative error incurred by approximating $\HSIC(X,Y)$ when $X_A$ is excluded from the input set. In contrast to the raw HSIC and distance correlation indices, the first-order and total HSIC sensitivity indices provide a meaningful ranking of parameter influence due to monotonicity under marginalization. Indeed, considering the effects of additional parameters upon the model output will only increase the corresponding sensitivity index. The following result demonstrates that the first-order and total HSIC sensitivity indices are bounded between 0 and 1, a key property not exhibited by the Sobol' indices under the assumption of arbitrarily-correlated inputs \cite{Hart18}.

\begin{lemma}\label{lem:bound}
	For any $A\subseteq \{1,\dots, p\}$, we have $0 \leq \mathbb{F}_A(Y)\leq 1$ and $0 \leq \mathbb{T}_A(Y)\leq 1$.
\end{lemma}
\begin{proof}
Immediately follows from the monotonicity of HSIC under marginalization. 
\end{proof}

In summary, the first-order and total HSIC sensitivity indices unify the strengths of existing
methods: they preserve the generality and computational efficiency of the raw HSIC, while
recovering the interpretability that make Sobol’ indices so
useful in practice. As such, the indices $\mathbb{F}_A(Y)$ and $\mathbb{T}_A(Y)$ provide compelling candidates for robust
sensitivity analysis across a wide variety of models. In the following section,
we describe how to estimate these sensitivity indices from data, and then present
numerical experiments comparing distance correlation, first-order HSIC, total HSIC, and total-effect Sobol’ indices in practice.

\subsection{Computing the sensitivity indices}\label{sec:theory}
We now turn to the numerical estimation of the first-order and total HSIC sensitivity indices defined in the previous section. 
Throughout the remainder of the article, we employ the augmented kernels prescribed in Section~\ref{sec:theory}, which were chosen to ensure that the indices satisfy monotonicity under marginalization and are characteristic on every separable Hilbert space. In particular, we choose Gaussian kernels as the base from which we obtain the augmented kernels. For each index subset $A\subset\{1,\dots,p\}$, we select the bandwidth parameter $\sigma_{X_A}$ for the Gaussian kernel $k_{\sigma_{X_A}}$ according to the sample-based {\em median heuristic} given by
\begin{equation}
  \sigma_{X_A} = \mathrm{median}\{\|X_A^{(i)} - X_A^{(j)}\|_{\cX_A} \ \vert \ {i<j}\},
  \label{eq:median_heuristic}
\end{equation}
computed over all pairwise distances among the samples. This heuristic provides a convenient baseline that adapts to the spread of the data and generally yields stable estimates near its chosen scale. As we demonstrate in the discussions to follow, this choice is sufficient for obtaining consistent rankings of parameter importance in the test problems considered. Having addressed the issue of kernel selection, we turn our attention now to the empirical estimator for the HSIC.

Given \( n \) independent realizations \((X^{(j)}, Y^{(j)})_{j=1}^n\) from the joint law of the model inputs and output, the well-known empirical estimator for the HSIC presented in~\cite{Gretton05a} takes the form\begin{equation}\label{eq:est}
  \widehat{\HSIC}(X, Y)
  = \frac{1}{n^2} \operatorname{tr}(\mathbf{KHLH}),
\end{equation}
where \( \mathbf{K, L} \in \mathbb{R}^{n \times n} \) are Gram matrices whose elements are respectively given by
$$K_{ij} := k(X^{(i)}, X^{(j)})\quad  \text{and}\quad L_{ij} := \ell(Y^{(i)}, Y^{(j)}). $$
The matrix \( \mathbf{H} \) centers the data and takes the form $$\mathbf{H} := \mathbf{I}_n - \mathbf{zz}^T,$$ where $\mathbf{z} = \frac{1}{\sqrt{n}}\1_n$. An estimator for $\HSIC(X_A,Y)$ is defined analogously, using $k_A$ in place of $k$ for the entries of $\mathbf{K}$. This matrix formulation highlights the computational structure of the estimator: once \( \mathbf{K} \) and \( \mathbf{L} \) are formed, \( \widehat{\mathrm{HSIC}} \) reduces to a single trace operation. 

Direct computation of the estimate in \cref{eq:est} involves an $\mathcal{O}(n^3)$ matrix product, which becomes computationally prohibitive for large sample sizes $n$. To achieve an $\mathcal{O}(n^2)$ complexity, the trace can be re-expressed into a simplified component-wise summation. This procedure avoids multiplication with the centering matrix $\mathbf{H}$, resulting in a faster calculation suitable for large-scale applications. The following result underpins our improved calculation.

\begin{lemma}\label{lem:trace}
Let $\mathbf{z} = \frac{1}{\sqrt{n}}\1_n$. We obtain the following identity: $$\operatorname{tr}(\mathbf{KHLH}) = \operatorname{tr}(\mathbf{KL}) - 2\langle \mathbf{Kz}, \mathbf{Lz}\rangle_2 + \langle \mathbf{z}, \mathbf{Kz}\rangle_2\langle \mathbf{z}, \mathbf{Lz}\rangle_2,$$ where $\langle \cdot, \cdot\rangle_2$ denotes the Euclidean inner product in $\bR^n$.
\end{lemma} 
\begin{proof}
This identity is a routine calculation using the cyclic property of the trace.
\end{proof}

Using this result, we can outline a simple procedure for empirically estimating the HSIC.
Letting $\mathbf{K}^{(i)}$ denote the $i^{\text{th}}$ column of $\mathbf{K}$ and using an analogous notation for the
columns of $\mathbf{L}$, we make two key observations. First, we recall that $\mathbf{K}$ and $\mathbf{L}$ are symmetric real-valued matrices, and so $\mathbf{K}^{(i)}$ and $\mathbf{L}^{(i)}$ are the same as the $i^{\text{th}}$ rows of the respective matrices. Second, we note that the $i^{\text{th}}$ row sum of $\mathbf{K}$ is equal to the $i^{\text{th}}$ element of $\mathbf{K1}_n$; that is, \begin{equation}
	[\mathbf{K1}_n]_i = \sum_{j=1}^n K_{ij}.
\end{equation}
An analogous identity holds for $\mathbf{L}$. With these remarks in mind, the Algorithm~\ref{alg:hsic} describes a procedure that requires only matrix-vector products with $\mathbf{K}$ and $\mathbf{L}$. There is no need to form any matrix products.

\begin{algorithm}[htbp]
    \caption{Efficient Computation of $\text{HSIC}$ Estimator}\label{alg:hsic}
    \begin{algorithmic}[1]
        \State Set $S = 0$
        \For{$i = 1,\ldots, n$}
            \State $\mathbf{a} = \mathbf{K}^{(i)}$
            \State $\mathbf{b} = \mathbf{L}^{(i)}$
            \State $S = S + \langle \mathbf{a},\mathbf{b}\rangle_2$
            \State $u_i = \langle \mathbf{a}, \mathbf{z}\rangle_2$
            \State $v_i = \langle \mathbf{b}, \mathbf{z}\rangle_2$
        \EndFor
        
        \State \textbf{Return} $\widehat{\HSIC}=\frac{1}{n^2}[S - 2\langle \mathbf{u},\mathbf{v}\rangle_2 + \langle \mathbf{u},\mathbf{z}\rangle_2\langle \mathbf{v},\mathbf{z}\rangle_2]$.
    \end{algorithmic}
\end{algorithm}

This procedure only requires that we store one column of the matrices $\mathbf{K}$ and $\mathbf{L}$ at a time, thereby reducing the storage requirement from $\mathcal{O}(n^2)$ to $\mathcal{O}(n)$. Furthermore, Algorithm~\ref{alg:hsic} describes a single-loop Monte Carlo estimator, which significantly out-performs the double-loop Monte Carlo estimators typically employed in variance-based sensitivity analysis.

The empirical estimator for the HSIC underpins our definition of an empirical estimator for the first-order and total HSIC sensitivity indices. For each input subset \( A \subseteq \{1,\dots, p\} \), the empirical estimators for \( \mathbb{F}_A(Y) \) and \( \mathbb{T}_A(Y) \) are respectively given by \begin{equation}\label{eq:tifest} \widehat{\mathbb{F}}_A(Y) \defeq \frac{\widehat{\HSIC}(X_{ A},Y)}{\widehat{\HSIC}(X,Y)}, \quad \widehat{\mathbb{T}}_A(Y) \defeq 1 - \frac{\widehat{\HSIC}(X_{\sim A},Y)}{\widehat{\HSIC}(X,Y)}.\end{equation} The numerators and denominators in~\eqref{eq:tifest} converge to their population values at rate \(\mathcal{O}(n^{-1/2}) \), and their biases are each of order \( \mathcal{O}(n^{-1}) \) \cite{Gretton05b}. In practice, these asymptotic properties translate into smooth convergence of the estimated indices as \( n \) increases. While an unbiased estimator for the HSIC has been presented \cite{daVeiga21, Gretton05a}, the convenient expression given by \cref{eq:est} is preferred for the scope of our study. In the following section, we implement the estimator for the first-order and total HSIC sensitivity indices in several illustrative numerical experiments.

\section{Numerical Experiments}\label{sec:numerics} In this section, we demonstrate the utility of the first-order and total HSIC sensitivity indices with three numerical experiments. We first consider the Ishigami function, a common example of a model with independent inputs and a scalar output. The first-order and total-effect Sobol' indices for this model are known analytically, and we demonstrate that the corresponding HSIC sensitivity indices provide a ranking of parameter influence that is consistent with these values. 

We then present a portfolio model whose output remains scalar-valued, but whose inputs obey a nontrivial correlation structure. We demonstrate that, in this case, the total HSIC sensitivity indices provide a more accurate and interpretable ranking of parameter influence than the total-effect Sobol' indices for models with correlated inputs proposed by \cite{Hart18}. Moreover, we illustrate that the first-order and total HSIC indices can differ in the presence of correlated inputs.

Finally, we consider a model that describes the spread of cholera over the course of 300 weeks. The model response is functional-valued in this example, and we demonstrate that the resulting total HSIC sensitivity indices vary noticeably depending upon the assumed correlation structure of the inputs. We use this fact to justify the consideration of parameter correlations in real-world models, especially in the context of model reduction. 

\subsection{Ishigami function}\label{sec:ishigami}

To illustrate the practical computation of HSIC-based sensitivity indices, we consider the Ishigami function, 
a classical nonlinear test problem defined as
\begin{equation}
 Y = f(X_1, X_2, X_3) = \sin(X_1) + a\,\sin^2(X_2) + b\,X_3^4 \sin(X_1),
  \label{eq:ishigami}
\end{equation}
where \( a = 5 \), \( b = 0.1 \), and the inputs \( X_i \sim \mathcal{U}(-\pi, \pi) \) are mutually independent. 
This model is well-known for its strong nonlinearity and interaction effects: \( X_3 \) influences the output only through its interaction with \( X_1 \), while the contributions of \( X_2 \) are independent of any interactions with the other inputs. Since the inputs are mutually independent, the first-order and total HSIC sensitivity indices are identical to the first-order and total HSIC-ANOVA indices proposed in~\cite{daVeiga21}, respectively.

\begin{figure}[htbp]
     \centering
     \begin{subfigure}[b]{0.48\textwidth}
         \centering
         \includegraphics[width=\textwidth]{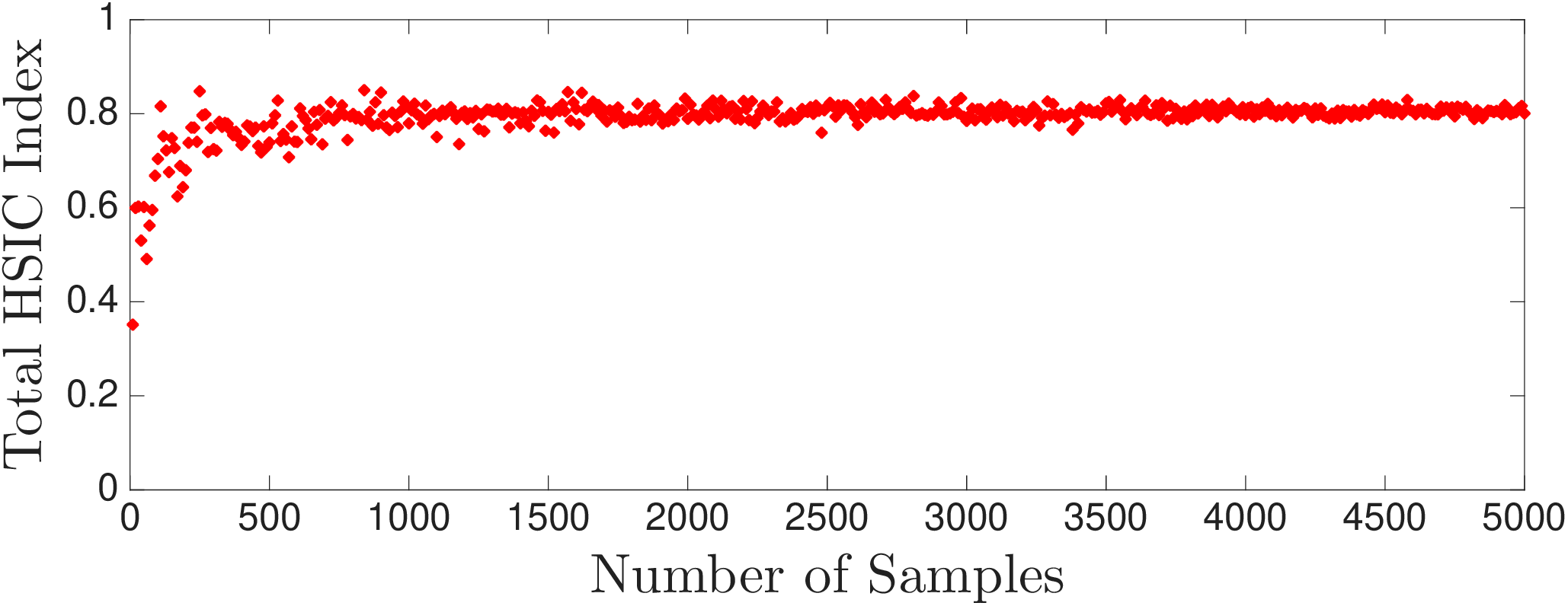}
         \caption{$\widehat{\mathbb{T}}_1(Y)$}
         \label{fig:T1}
     \end{subfigure}
     \hfill
     \begin{subfigure}[b]{0.48\textwidth}
         \centering
         \includegraphics[width=\textwidth]{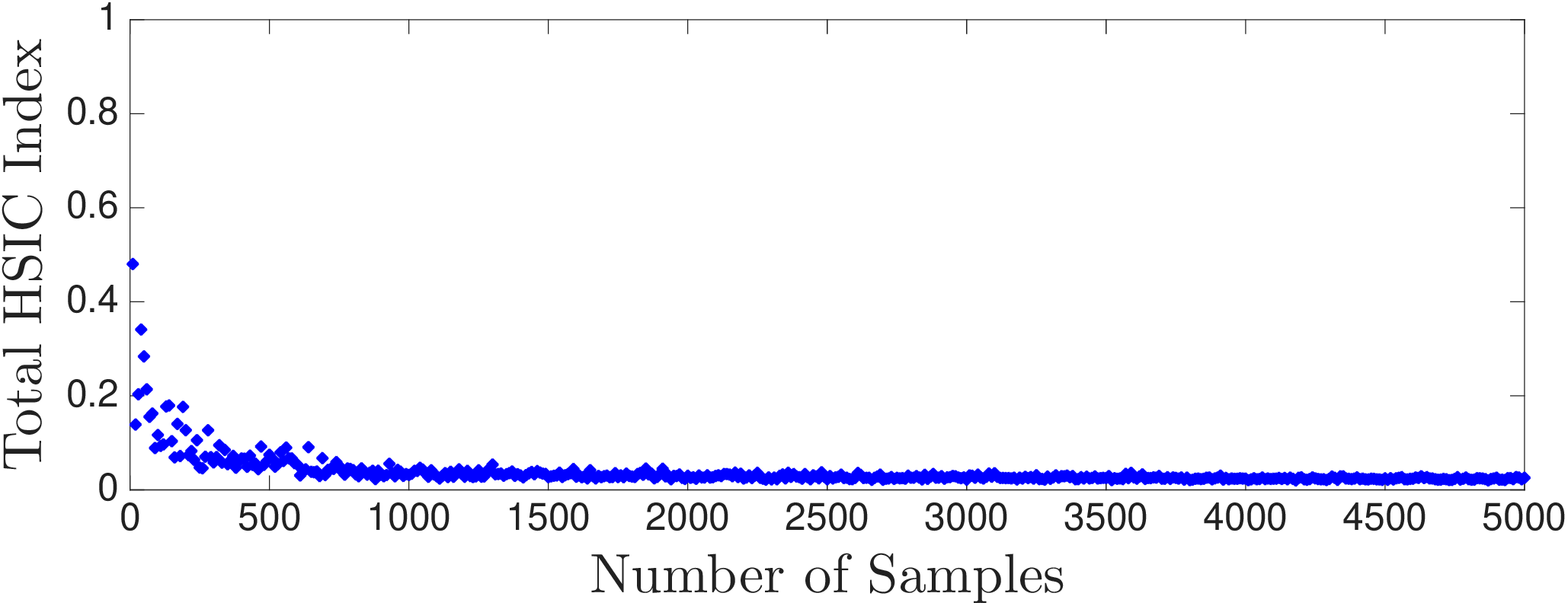}
         \caption{$\widehat{\mathbb{T}}_2(Y)$}
         \label{fig:T2}
     \end{subfigure}
     \hfill
     \begin{subfigure}[b]{0.48\textwidth}
         \centering
         \includegraphics[width=\textwidth]{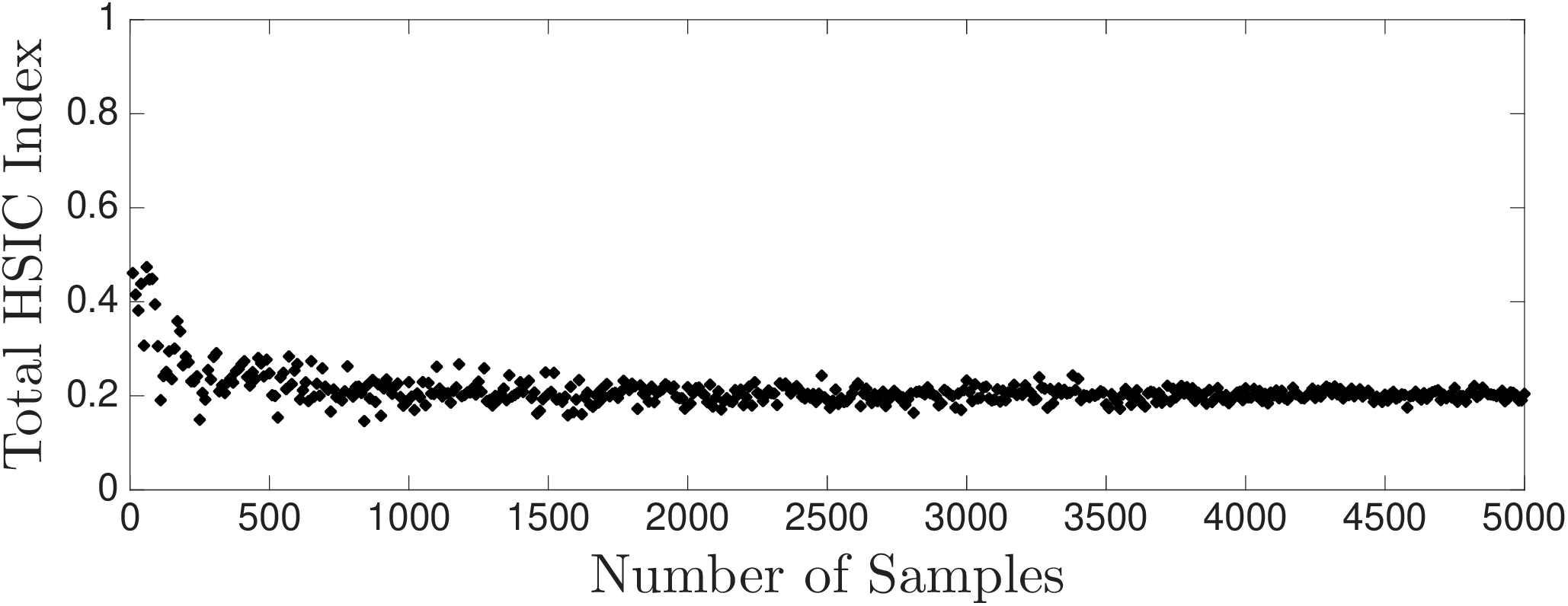}
         \caption{$\widehat{\mathbb{T}}_3(Y)$}
         \label{fig:T3}
     \end{subfigure}
        \caption{Estimated total HSIC sensitivity indices as functions of sample size $n$.}
        \label{fig:ishigamiconv}
\end{figure}

To assess sampling adequacy, Figure~\ref{fig:ishigamiconv} depicts the estimates $\widehat{\mathbb{T}}_i(Y)$ as a function of the sample size. The estimates stabilize after around \( n = 700 \) samples, indicating that any larger choice of \( n \) provides sufficient accuracy for all three input parameters. A similar computation shows that the estimates $\widehat{\mathbb{F}}_i(Y)$ stabilize after around \( n = 600 \) samples. Given this analysis, we draw \( n = 1,000 \) independent samples of \( (X_1, X_2, X_3) \) and evaluate \( Y = f(X_1, X_2, X_3) \). 
For each input \( X_i \), first-order and total HSIC indices are computed using the estimator in \cref{eq:tifest} with the kernels as prescribed in Section~\ref{sec:theory} and bandwidths determined by the median heuristic~\cref{eq:median_heuristic}. 
Distance correlation indices and total-effect Sobol' indices \( S_i^T \) are also computed for comparison. Figure~\ref{fig:ishigami} depicts the resulting indices. 

\begin{figure}[htbp]
     \centering
      \begin{subfigure}[htbp]{0.49\textwidth}\label{fig:ishigami}
         \centering
         \includegraphics[width=0.8\textwidth]{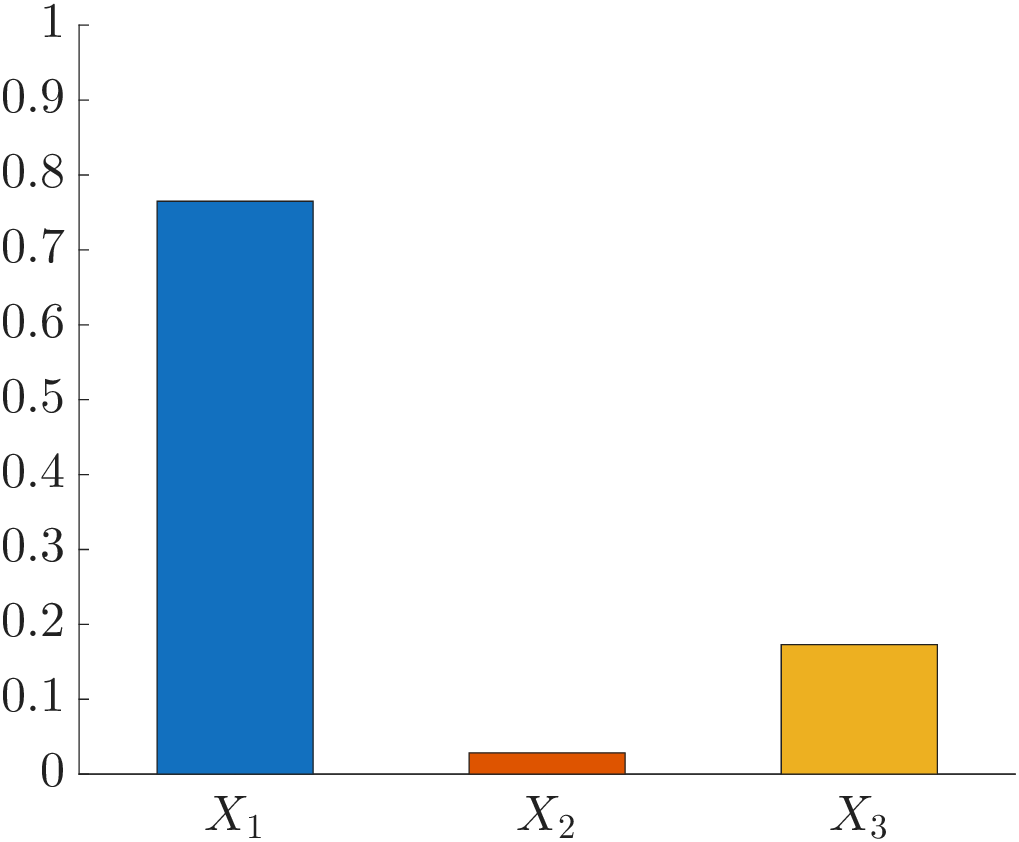}
         \caption{First-Order HSIC Indices}
         \label{fig:IshigamiFirst}
     \end{subfigure}
     \hfill
     \begin{subfigure}[htbp]{0.49\textwidth}
         \centering
         \includegraphics[width=0.8\textwidth]{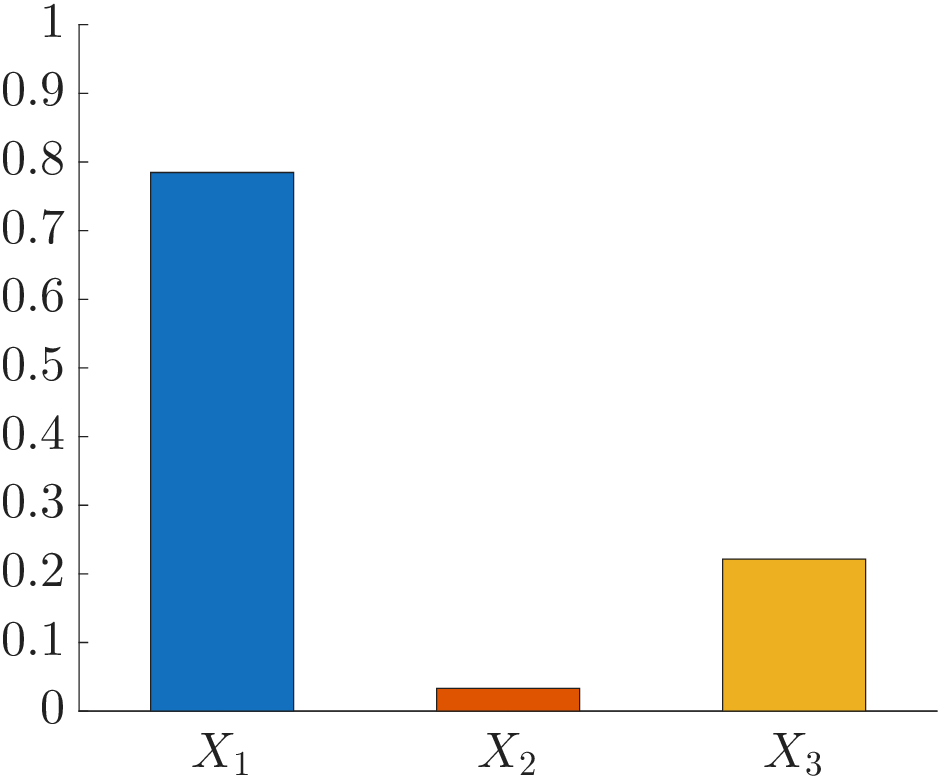}
         \caption{Total HSIC Indices}
         \label{fig:IshigamiTotal}
     \end{subfigure}

     \begin{subfigure}[htbp]{0.49\textwidth}
         \centering
         \includegraphics[width=0.8\textwidth]{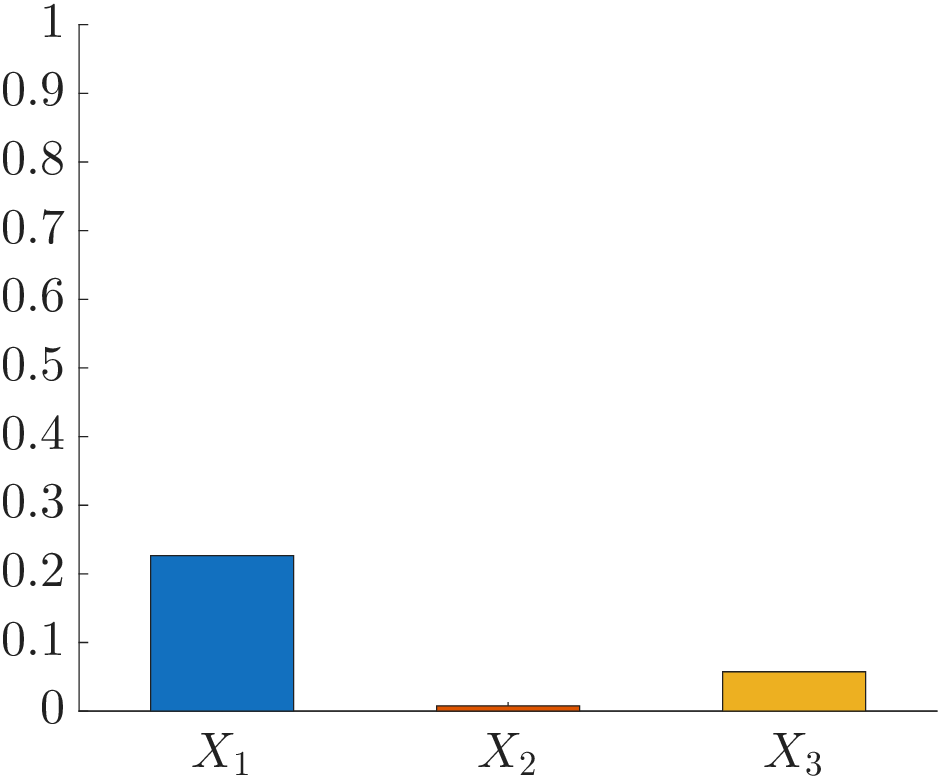}
         \caption{Distance Correlation Indices}
         \label{fig:IshigamiDcorr}
     \end{subfigure}
     \hfill
     \begin{subfigure}[htbp]{0.49\textwidth}
         \centering
         \includegraphics[width=0.8\textwidth]{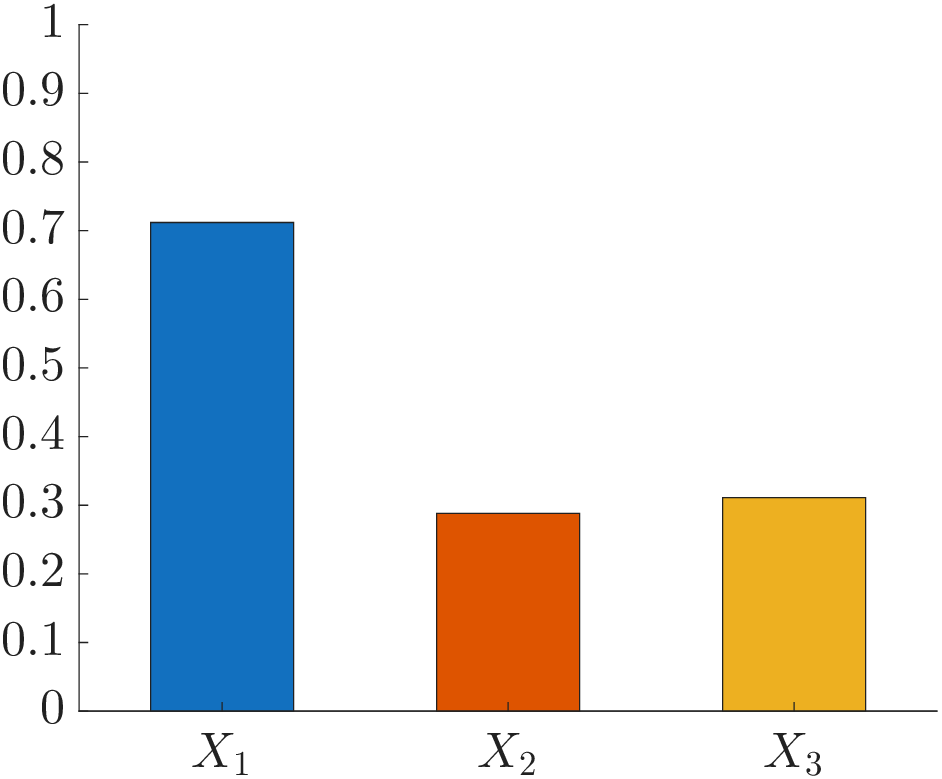}
         \caption{Total-effect Sobol' Indices}
         \label{fig:IshigamiSobol}
     \end{subfigure}
        \caption{A comparison of sensitivity indices for the Ishigami function.}
        \label{fig:ishigami}
\end{figure}

All four indices provide a consistent ranking of parameter influence in Figure~\ref{fig:ishigami}. This agreement reflects the fact that, under independent inputs, the HSIC- and variance-based measures capture similar notions of importance: while Sobol' indices quantify contributions to output variance, the HSIC indices detect general statistical dependence. The magnitudes of the distance correlation indices differ substantially from those of the other indices due to normalization differences inherent to the metric. Among the three other indices, any discrepancies in magnitude are attributable to the HSIC's moment-independence, which quantifies dependence structures beyond second-order moments. 

This experiment highlights that the first-order and total HSIC sensitivity indices retain the interpretability of variance-based measures while considering higher-order dependences. The magnitudes of these two indices differ only slightly in Figure~\ref{fig:ishigami}, agreeing with the results first produced in~\cite[Section 5.1]{daVeiga21}. It is suggested in that paper that total HSIC indices provide additional information over first-order ones only for very specific cases. Our work provides further evidence for this claim, at least for models with independent inputs. Next, we investigate the performance of these indices for models with dependent inputs.

\subsection{Correlated portfolio model}\label{sec:corrport}
In this example, we illustrate how the first-order and total HSIC sensitivity indices behave when
the assumption of parameter independence is removed. We examine a model that was
first considered in~\cite{Hart18} for their study of  Sobol' indices for
models with correlated inputs. The model is defined by
\begin{equation}\label{eq:corrport} 
    Y = 20X_1 + 16X_2 + 12X_3 + 10X_4 + 4X_5,
\end{equation}
where $X = [X_1,\dots,X_5]$ follows a multivariate normal distribution $\mathcal N(\mu,\Sigma)$.

The mean vector $\mu$ and covariance matrix $\Sigma$ are given by
\begin{equation}
\mu =
\begin{bmatrix}
0\\0\\0\\0\\0
\end{bmatrix},
\qquad
\Sigma =
\begin{bmatrix}
1 & 0.5\rho & 0.5\rho & 0 & 0.8\rho\\
0.5\rho & 1 & 0 & 0 & 0\\
0.5\rho & 0 & 1 & 0 & 0.3\rho\\
0 & 0 & 0 & 1 & 0\\
0.8\rho & 0 & 0.3\rho & 0 & 1
\end{bmatrix}.
\end{equation} In this example, $\rho\in[0,1]$ controls the strength of correlation among
subsets of the inputs.  

It was noted in~\cite[Figure 1]{Hart18} that as $\rho$
increases toward $1$, the indices corresponding to the four correlated
parameters collapse and the only uncorrelated variable, $X_4$, appears to be
dominant.  The authors explain this behavior as follows: when correlations
between parameters are strengthened, the influence of one variable upon the
output may be approximated by the other variables. This phenomenon does not
exist in the case of independent inputs, but explains why the total-effect
Sobol' indices are decreasing in $\rho$ for all parameters except for $X_4$. 
Despite this explanation, we demonstrate that Sobol' analysis fails to provide
an accurate parameter ranking in the case of correlated inputs, when the goal is
model reduction by fixing non-influential parameters. 

\begin{figure}[htbp]
     \centering
     \begin{subfigure}[htbp]{0.48\textwidth}
         \centering
         \includegraphics[width=\textwidth]{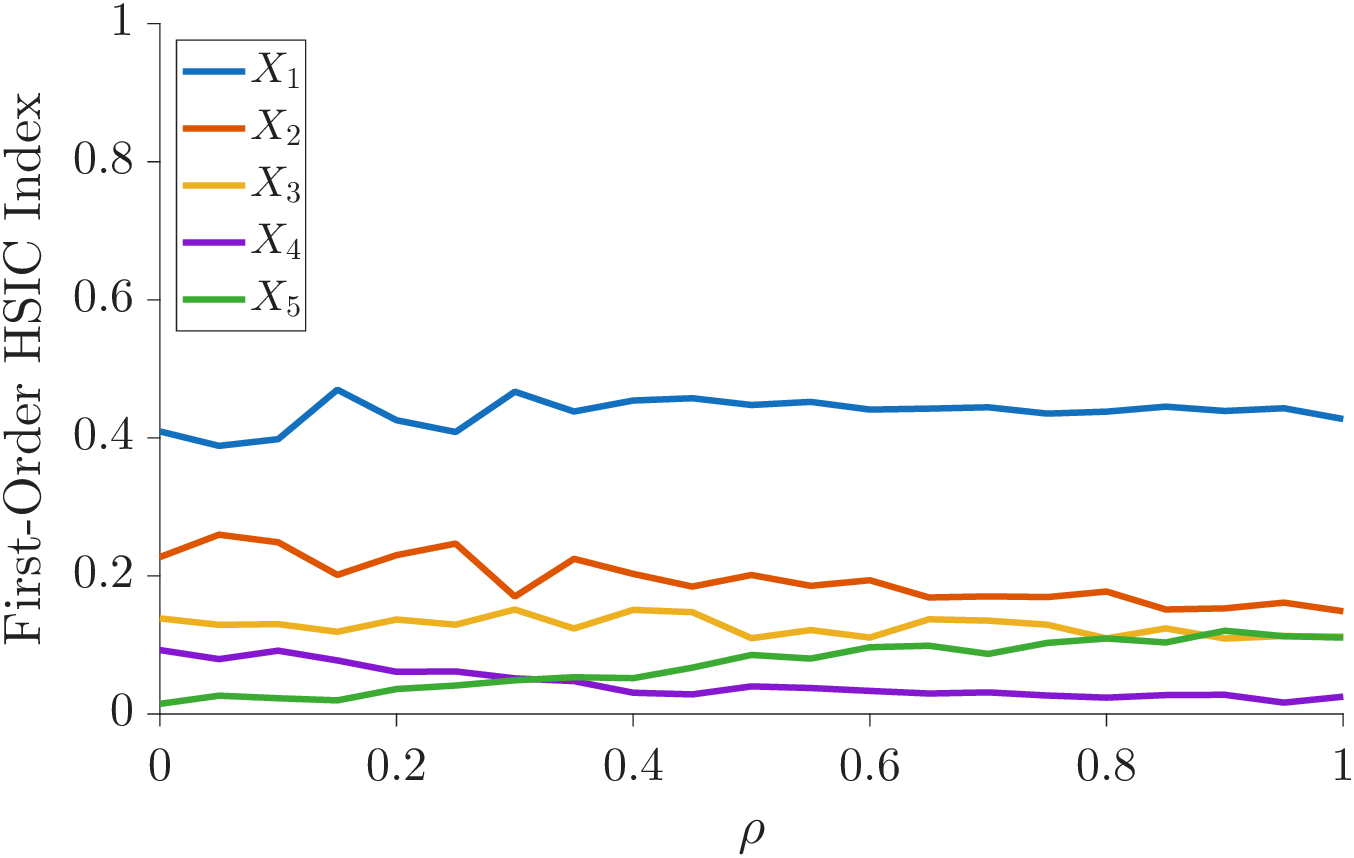}
     \end{subfigure}   
     \hfill \begin{subfigure}[htbp]{0.48\textwidth}
         \centering
         \includegraphics[width=\textwidth]{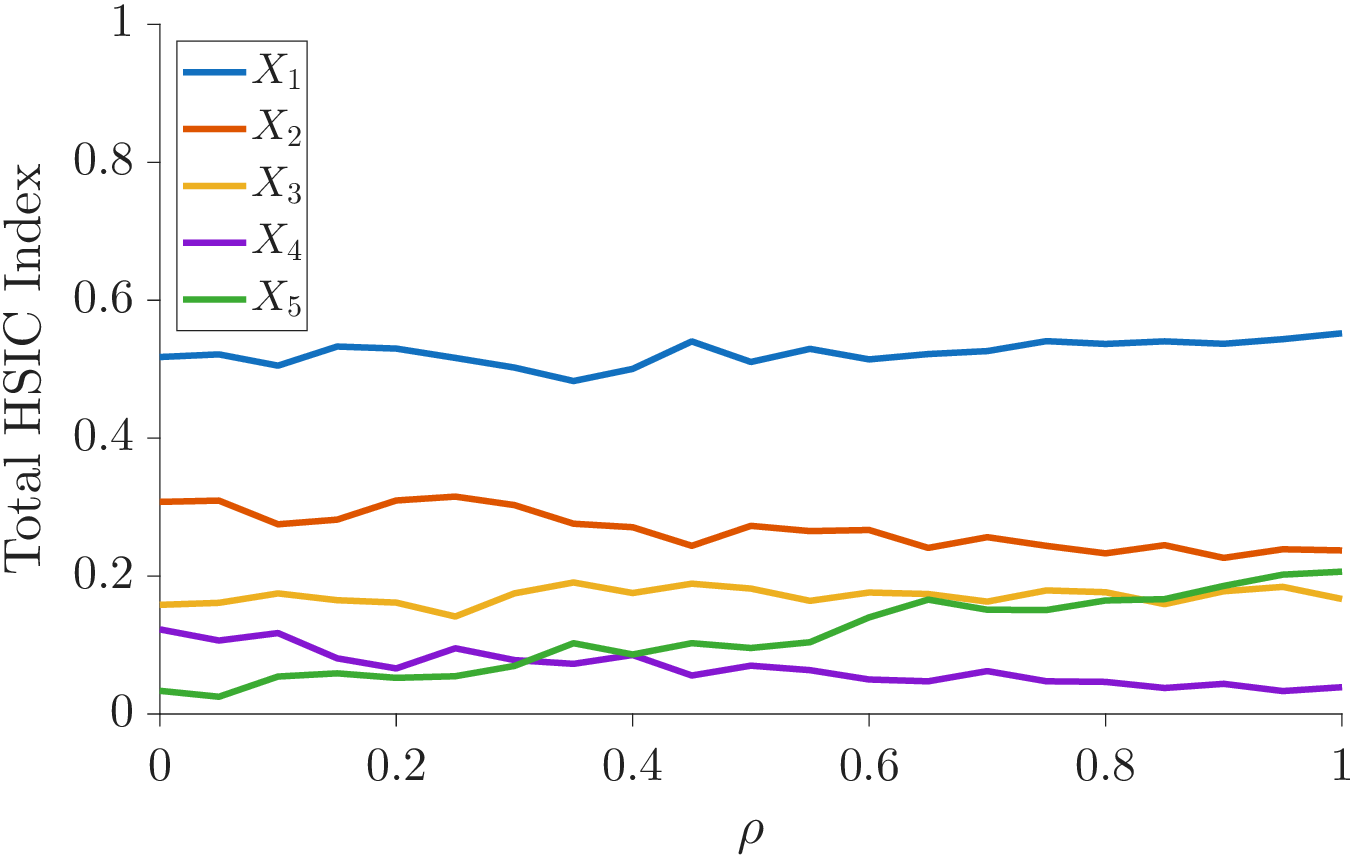}
     \end{subfigure}       
     \begin{subfigure}[htbp]{0.48\textwidth}
         \centering
         \includegraphics[width=\textwidth]{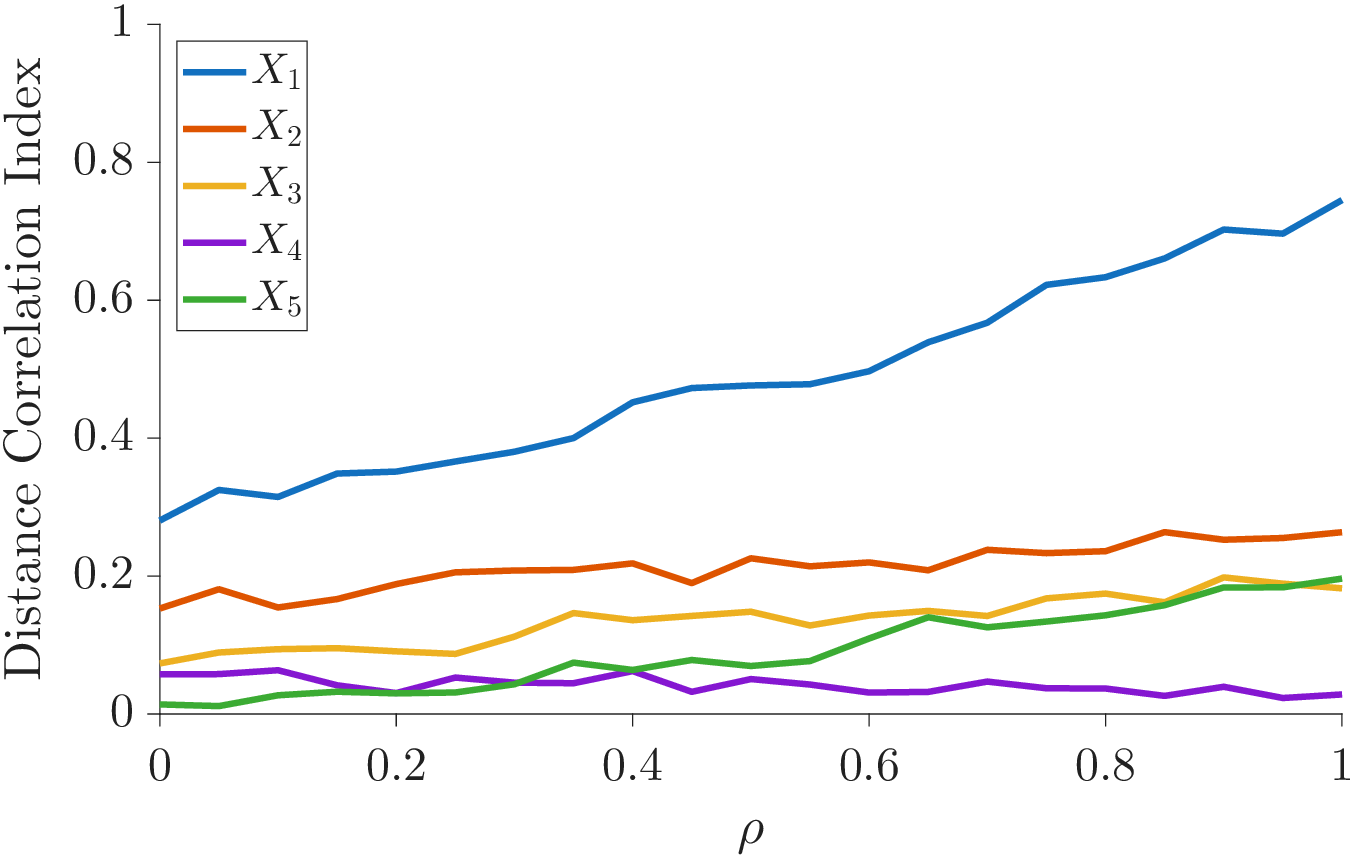}
     \end{subfigure}
     \caption{Sensitivity indices for the correlated portfolio model as functions of $\rho$.}
        \label{fig:corr}
\end{figure}

We compute distance correlation, first-order, and total HSIC sensitivity indices
for the correlation values $\rho_j = j/20$ for  $j=0,\dots,20$, using $n=2,000$
samples per computation. The resulting indices are plotted in
Figure~\ref{fig:corr} as functions of $\rho$. Unlike the Sobol' or distance
correlation indices, the HSIC-based indices remain stable across increasing
correlation. The ranking of influential variables is preserved for all $\rho$:
$X_1$ remains most influential, while $X_5$ gradually gains importance as its
correlations with the dominant variables strengthen. At the high-correlation
extreme ($\rho\approx1$), total HSIC indices indicate that $X_5$ surpasses $X_4$
and $X_3$ in importance, consistent with the model structure. The total HSIC
index is well-suited for model reduction, but does not always provide more
information than its first-order counterpart. 
We also point to the work~\cite{daVeiga21}, which introduces Shapley-HSIC
indices for the purposes of model analysis, i.e., understanding the distinction 
of first, mixed, and total effects on the output.

It should be noted that while the distance correlation and HSIC-based indices provide the same parameter rankings, the magnitudes of these indices are only interpretable in the case of first-order and total HSIC. Moreover, we notice that the first-order and total HSIC indices differ significantly in magnitude, which contrasts their typical agreement in models with independent inputs. In models with correlated inputs, the total HSIC indices aggregate more interactions than the first-order HSIC indices, providing a more robust ranking of parameter influence (compare, e.g., the values of $\widehat{\mathbb{F}}_3(Y)$ and $\widehat{\mathbb{F}}_5(Y)$ with $\widehat{\mathbb{T}}_3(Y)$ and $\widehat{\mathbb{T}}_5(Y)$ at $\rho=1$). 

To qualitatively validate these findings, we performed parameter dimension reduction at $\rho=0$ and $\rho=1$ by fixing the least influential parameter ($X_5$ and $X_4$, respectively, as indicated by Figure~\ref{fig:corr}) at its mean value. We compare the resulting conditional output distribution to that of the full model using $n=1\times 10^7$ Monte Carlo samples. As shown in Figure~\ref{fig:corrportMR}, the reduced and full output distributions agree closely in both cases, supporting the conclusion that total HSIC sensitivity indices correctly identify parameters whose uncertainty can be neglected without materially affecting the model response.

\begin{figure}[htbp]
     \centering
     \begin{subfigure}[b]{0.48\textwidth}
         \centering
         \includegraphics[width=\textwidth]{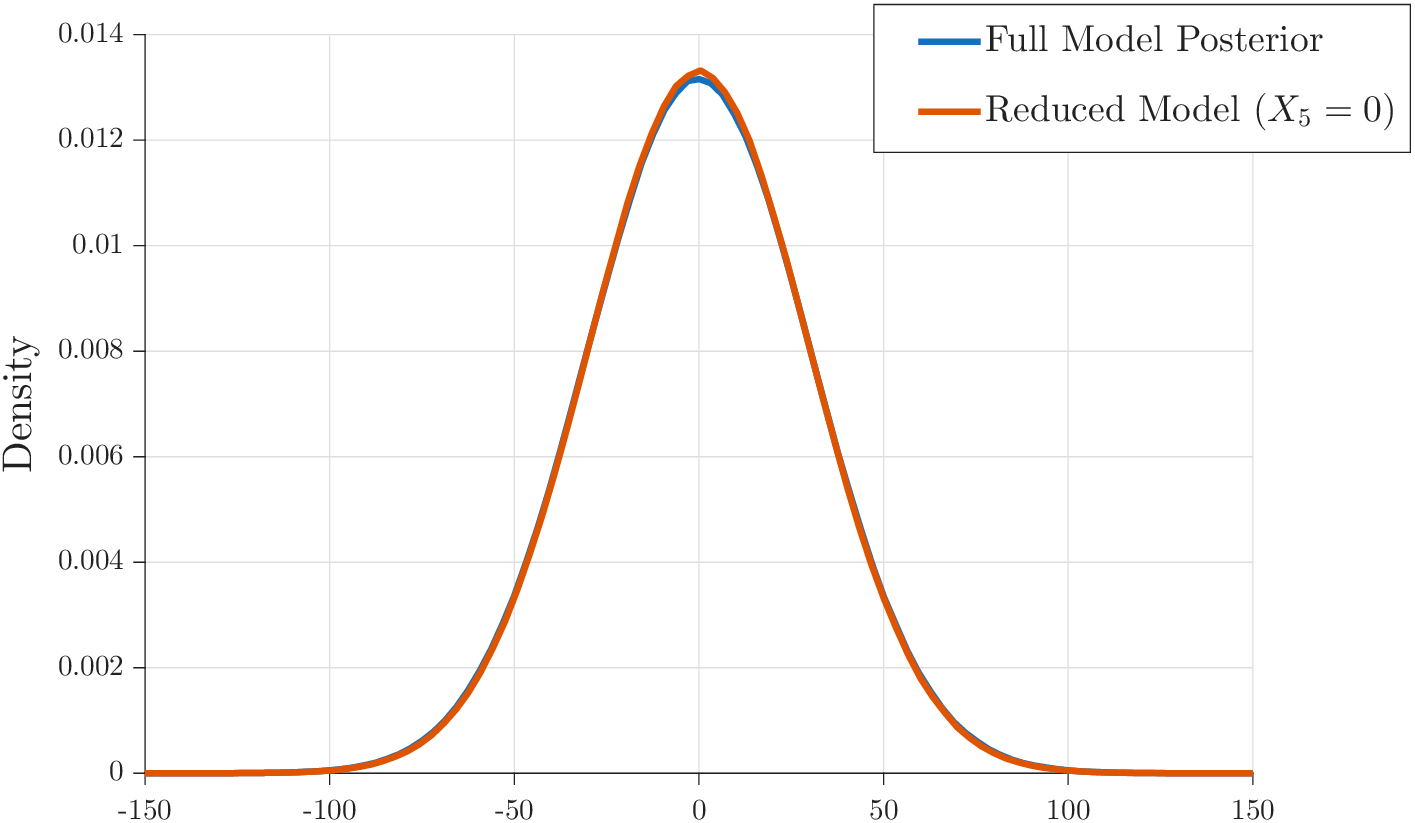}
         \caption{$\rho = 0$}
         \label{fig:rho0MR}
     \end{subfigure}
     \hfill
     \begin{subfigure}[b]{0.48\textwidth}
         \centering
         \includegraphics[width=\textwidth]{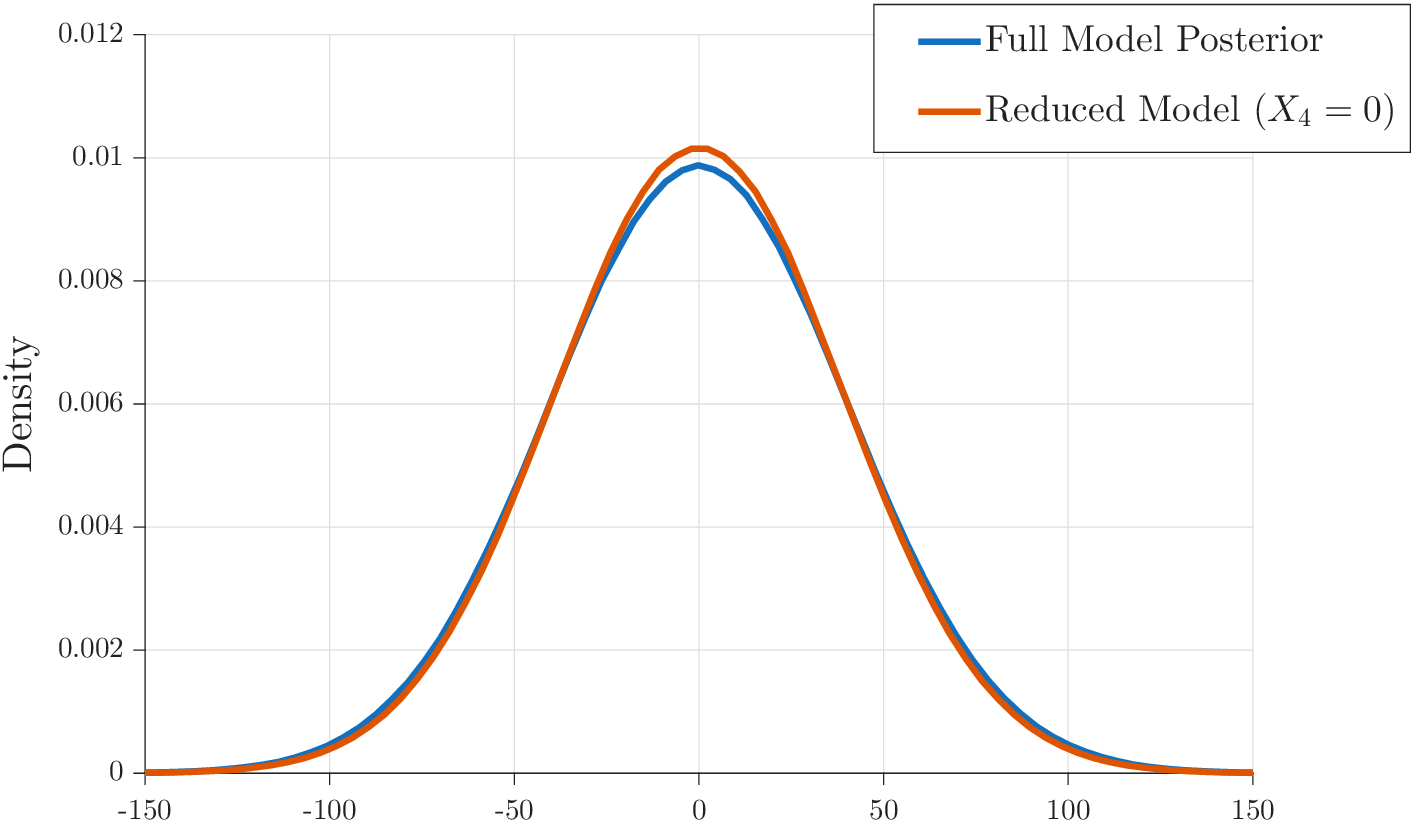}
         \caption{$\rho = 1$}
         \label{fig:rho1MR}
     \end{subfigure}        
     \caption{Model reduction informed by Figure~\ref{fig:corr} at different values of $\rho$.}
        \label{fig:corrportMR}
\end{figure}

\begin{figure}[htbp]
	\centering
	\includegraphics[width=0.55\textwidth]{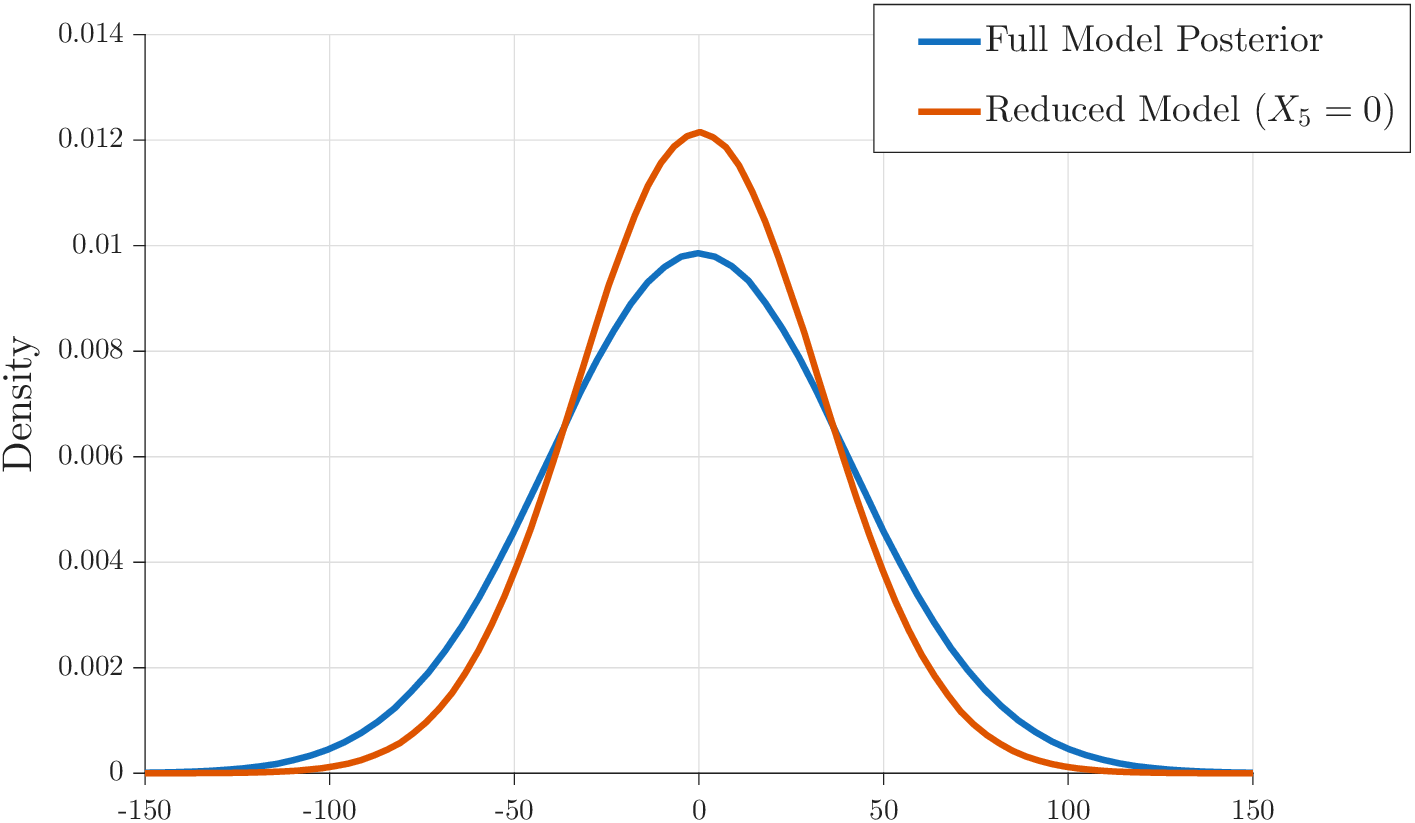}
	\caption{Model reduction at $\rho=1$ informed by the approach in~\cite{Hart18}.}
	\label{fig:rho1MRFalse}
\end{figure}

We also used this method to demonstrate the limitations of the total-effect
Sobol' indices for parameter dimension reduction. Indeed, \cite{Hart18} suggests that $X_5$
holds minimal influence over the model output when $\rho=1$.  This may suggest 
$X_5$ as a candidate to be fixed at its mean value for the purpose of dimension
reduction. We compare the conditional output distribution to that of the full
model using $n=1\times 10^7$ Monte Carlo samples in Figure~\ref{fig:rho1MRFalse}. The
significant discrepancy between the distributions suggests that $X_5$ holds more
influence over the model than the Sobol' indices describe, further demonstrating
the advantage of total HSIC sensitivity indices over variance-based sensitivity
measures. 

Overall, this example demonstrates that total HSIC indices preserve interpretability and ranking consistency under correlated inputs. They correctly identify parameter importance without being misled by shared variability, offering a principled and robust alternative to variance-based sensitivity measures.

\subsection{Function-valued cholera model}\label{sec:cholera}
As a final example, we apply the total HSIC framework to a model with correlated
inputs and a function-valued output.  Specifically, we consider the nonlinear
epidemiological model of cholera transmission introduced in \cite{Cholera} and
studied in \cite{Alexanderian20, Cleaves19}. This model describes the coupled
dynamics between the human population and two environmental bacterial
reservoirs, capturing both direct and indirect infection pathways. The objective
is to quantify the sensitivity of the infected population $I(t)$ to
uncertainties in the epidemiological parameters, and to evaluate whether
HSIC-based indices remain robust in this biologically realistic,
correlated-input setting.

The population is divided into susceptible ($S$), infectious ($I$), and recovered ($R$) compartments, and the environmental bacteria are split into highly infectious ($B_H$) and lowly infectious ($B_L$) concentrations. The state vector $y = (S, I, R, B_H, B_L)^\top \in \mathbb{R}^5$ evolves according to
\begin{equation}\label{eq:cholera}\begin{split}
\frac{dS}{dt} &= b N_{\mathrm{pop}} 
- \beta_L S \frac{B_L}{\kappa_L + B_L} 
- \beta_H S \frac{B_H}{\kappa_H + B_H}
- bS, \nonumber\\[2pt]
\frac{dI}{dt} &= 
\beta_L S \frac{B_L}{\kappa_L + B_L} 
+ \beta_H S \frac{B_H}{\kappa_H + B_H}
- (\gamma + b) I, \nonumber\\[2pt]
\frac{dR}{dt} &= \gamma I - bR, \nonumber\\[2pt]
\frac{dB_H}{dt} &= \xi I - \chi B_H, \nonumber\\[2pt]
\frac{dB_L}{dt} &= \chi B_H - \delta B_L. \end{split}
\end{equation}
Initial conditions are given by $(S(0), I(0), R(0), B_H(0), B_L(0)) = (N_{\mathrm{pop}}-1, 1, 0, 0, 0)$, with $N_{\mathrm{pop}} = 10,000$, and the system is solved up to $T=300$ weeks using MATLAB’s \texttt{ode45} solver. The model parameters, along with their units and nominal values adopted from \cite{Cleaves19, Cholera}, are summarized in Table 1.

\begin{table}[htbp]
\begin{tabular}{llll}
\toprule
Parameter & Description & Units & Nominal value \\
\midrule
$\beta_L$ & Rate of drinking low-infectious cholera & week$^{-1}$ & 1.5 \\
$\beta_H$ & Rate of drinking high-infectious cholera & week$^{-1}$ & 7.5 \\
$\kappa_L$ & $B_L$ cholera carrying capacity & bacteria·mL$^{-1}$ & $10^6$ \\
$\kappa_H$ & $B_H$ cholera carrying capacity & bacteria·mL$^{-1}$ & $7\times10^8$ \\
$b$ & Human birth/death rate & week$^{-1}$ & $1/1560$ \\
$\chi$ & Decay rate from $B_H$ to $B_L$ & week$^{-1}$ & 1/168 \\
$\xi$ & $B_H$ water contamination rate & $\frac{\text{bacteria·mL$^{-1}$}}{\text{individual·week}}$ & 70 \\
$\delta$ & Death rate of $B_L$ cholera & week$^{-1}$ & 7/30 \\
$\gamma$ & Recovery rate from cholera & week$^{-1}$ & 7/5 \\
\bottomrule
\end{tabular}\label{tab:cholera_params}
\caption{Nominal cholera model parameters from \cite{Cholera, Cleaves19}.}
\end{table}

We simulate a realistic study by obtaining a data-informed estimate of parameter correlations. As described in great detail by \cite{SmithUQ}, we perform an ordinary least squares (OLS) fit of the model to synthetic data to obtain estimates of the mean parameter vector and corresponding empirical correlation matrix. The correlation structure, depicted in Figure~\ref{fig:cholera_corr}, reveals strong dependencies among the transmission and bacterial parameters, particularly $(\beta_H, \beta_L)$ and $(\xi, \chi)$, corresponding to shared infection and decay processes.

\begin{figure}[htbp]
    \centering
    \includegraphics[width=0.47\textwidth]{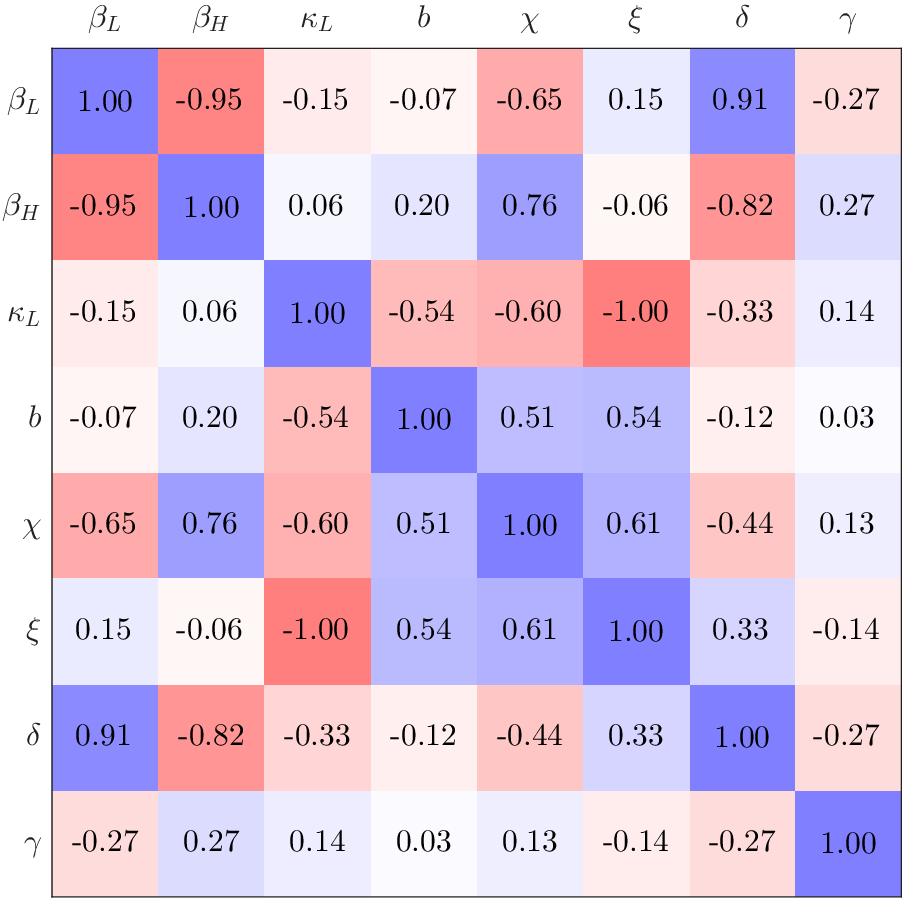}
    \caption{Empirical parameter correlation matrix obtained from OLS estimation. Strong correlations occur between parameters governing infection and bacterial dynamics.}
    \label{fig:cholera_corr}
\end{figure}

The quantity of interest is the infected population $I(t)$ over $t \in [0,300]$. We compute the total HSIC indices for each parameter using $n=1,500$ samples drawn from the multivariate normal distribution with mean vector and covariance matrix obtained through the OLS estimates computed above. As a comparison, we also compute the total HSIC indices using $n=1,500$ samples drawn from independent uniform distributions with bounds determined by $\pm10\%$ of the mean vector determined by the OLS estimate. The subsequent indices vary significantly, and the results are displayed in Figure~\ref{fig:cholindices}. 

\begin{figure}[htbp]
     \centering
     \begin{subfigure}[b]{0.48\textwidth}
         \centering
         \includegraphics[width=\textwidth]{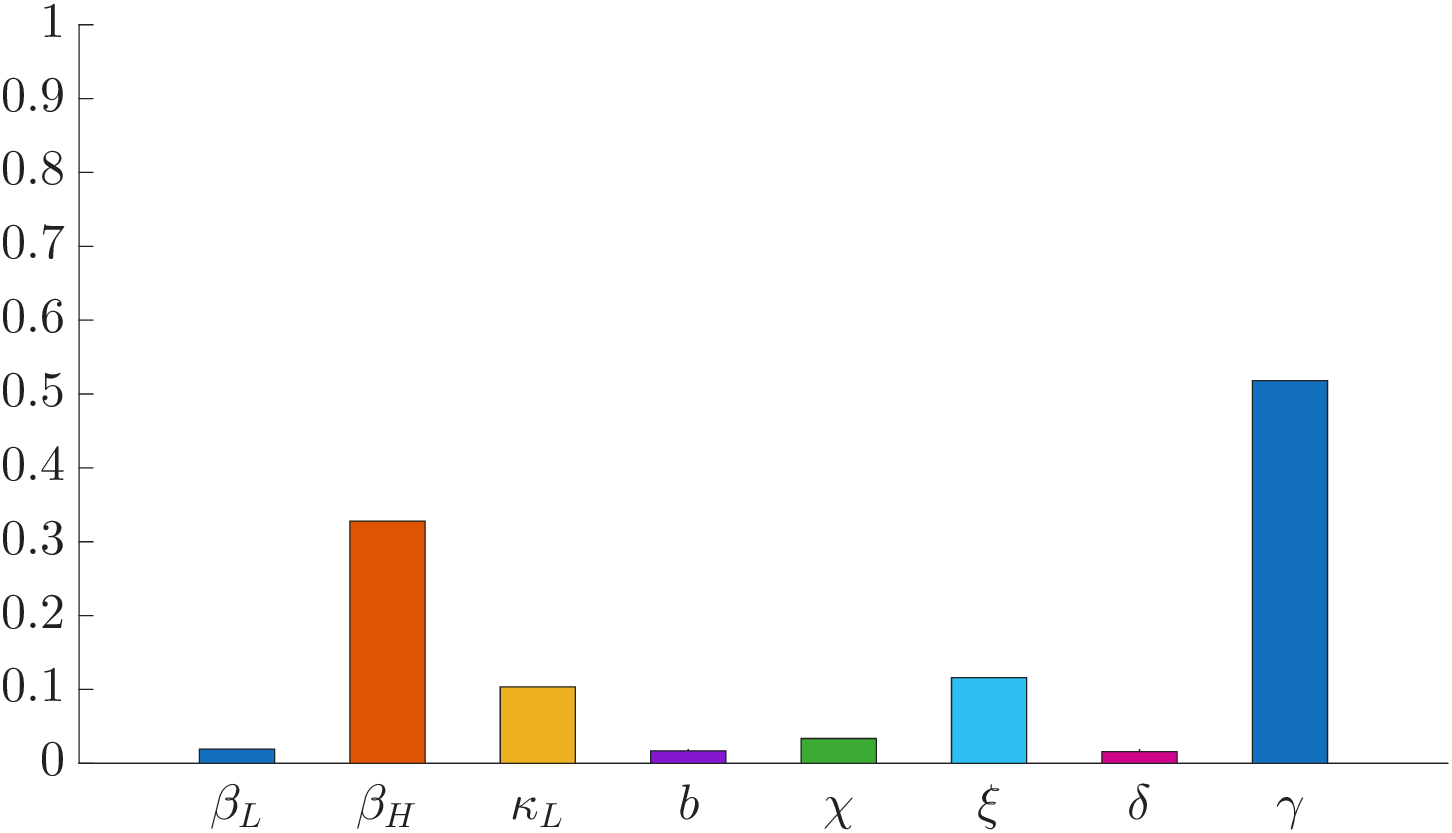}
         \caption{Independent uniform sampling}
         \label{fig:choluniform}
     \end{subfigure}
     \hfill
     \begin{subfigure}[b]{0.48\textwidth}
         \centering
         \includegraphics[width=\textwidth]{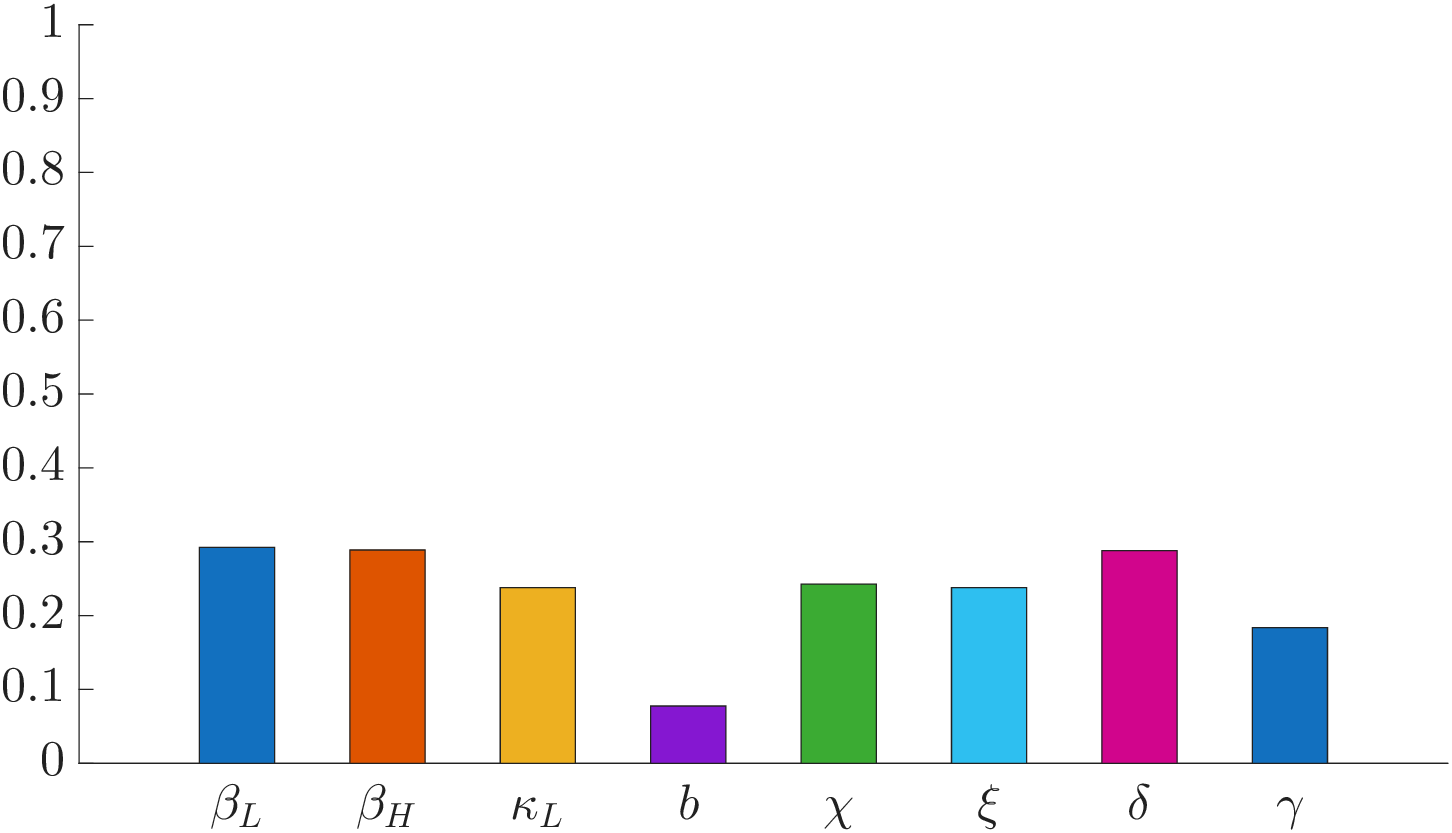}
         \caption{Multivariate normal sampling}
         \label{fig:cholcorr}
     \end{subfigure}        
     \caption{A comparison of total HSIC indices under the assumptions of independent (Figure~\ref{fig:choluniform}) and correlated (Figure~\ref{fig:cholcorr}) parameters.}
        \label{fig:cholindices}
\end{figure}

Both sets of total HSIC indices identify $\beta_H$ as significantly influencing the infection curve $I(t)$. Similarly, both sets of indices capture the minimal effects of $b$ upon the model output. The remaining parameters' influence varies significantly between the assumptions of independent and correlated inputs, reinforcing our claim that the parameter correlation structure should not be neglected when performing sensitivity analysis. In particular, the correlation structure greatly influences the feasibility of parameter dimension reduction based on the total HSIC sensitivity indices. 

\begin{figure}[htbp]
    \centering
    \includegraphics[width=0.7\textwidth]{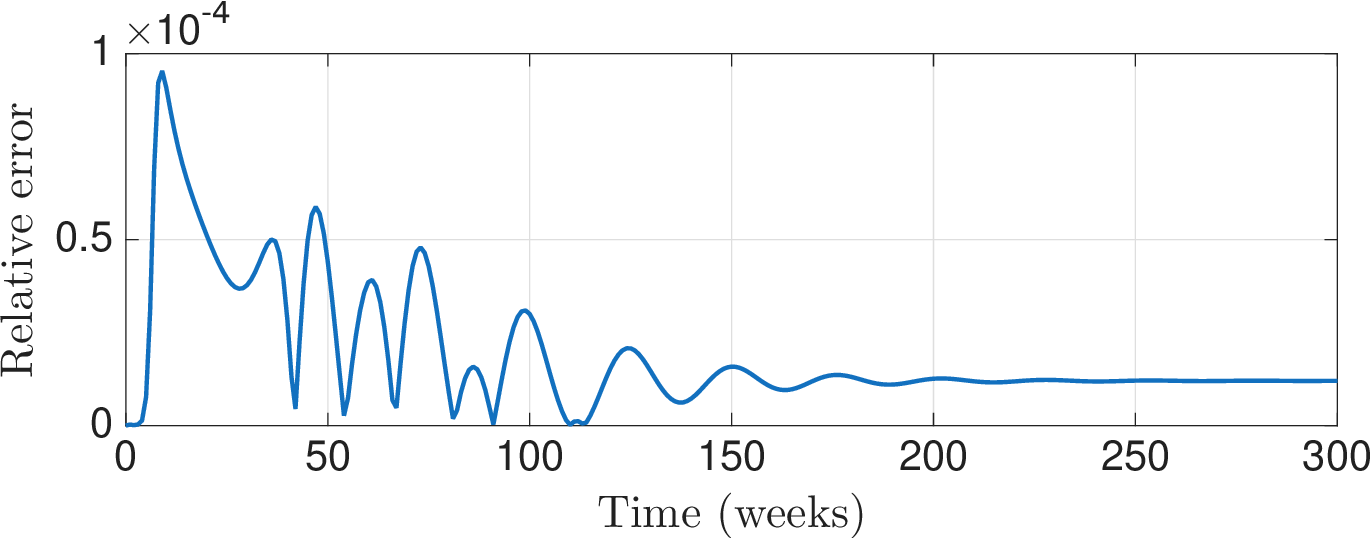}
    \caption{Relative error between the full and reduced posterior distributions of $I(t)$ after fixing $b$ at its mean value. The small error supports the use of HSIC-guided model reduction.}
    \label{fig:cholera_error}
\end{figure}

Figure~\ref{fig:cholcorr} identifies $b$ as the least influential parameter. We
therefore fix $b$ at its OLS mean value and generate the full and reduced
distributions of the model output using $n=1\times10^6$ Monte Carlo samples.
Figure~\ref{fig:cholera_error} shows the relative error between the mean curves
of the full and reduced distributions of $I(t)$, demonstrating that fixing $b$
introduces negligible error. This result supports the exclusion of $b$ from the
model for simplified inference.  The present experiment demonstrates that total
HSIC indices extend naturally to time-dependent and correlated epidemiological
models, yielding interpretable and stable sensitivity rankings that can guide
principled model reduction. 

\section{Conclusions}\label{sec:conc}
The augmented kernel construction presented in this work guarantees the property
of monotonicity under marginalization for HSIC-based sensitivity analysis. 
Under this framework, the first-order and total HSIC sensitivity indices accommodate a wide
variety of parameter types, are interpretable as a share of total influence over
the model output, and require no assumptions of independence among the inputs. 
Our numerical experiments demonstrate that the HSIC-based indices
correctly identify non-influential parameters in the presence of strong
statistical dependence among inputs.  As such, these 
indices can serve as flexible tools for global sensitivity analysis.

The estimation of HSIC-based indices for high-dimensional models is an important
line of future inquiry. Some studies have reported degradations in the
statistical power of the HSIC for independence testing in problems with
high-dimensional parameters~\cite{Zhang2023,Zhu2020}. This was also noted in our
preliminary numerical tests with high-dimensional models in the context of
sensitivity analysis. Specifically, these tests revealed that
the median heuristic, which is a commonly used approach for kernel tuning, loses
its effectiveness for models with high-dimensional parameters. 
In our future work, we will investigate
suitable kernel tuning techniques and estimators to improve scalability of
HSIC-based indices in high-dimensional models.  A related idea is the use of
efficient surrogate models to relieve the sampling procedure from potentially
expensive model evaluations.  

Furthermore, we plan to investigate the decomposition of total HSIC indices into
component terms to achieve an attribution of influence analogous to the
HSIC-ANOVA decomposition, but without the assumption of parameter independence.


\bibliographystyle{siamplain}
\bibliography{thesisrefs}

\end{document}


\maketitle

\section{A detailed example}

Here we include some equations and theorem-like environments to show
how these are labeled in a supplement and can be referenced from the
main text.
Consider the following equation:
\begin{equation}
  \label{eq:suppa}
  a^2 + b^2 = c^2.
\end{equation}
You can also reference equations such as \cref{eq:matrices,eq:bb} 
from the main article in this supplement.

\lipsum[100-101]

\begin{theorem}
An example theorem.
\end{theorem}

\lipsum[102]
 
\begin{lemma}
An example lemma.
\end{lemma}

\lipsum[103-105]

Here is an example citation: \cite{KoMa14}.

\section[Proof of Thm]{Proof of \cref{thm:bigthm}}
\label{sec:proof}

\lipsum[106-112]

\section{Additional experimental results}
\Cref{tab:foo} shows additional
supporting evidence. 

\begin{table}[htbp]
\footnotesize
  \caption{Example table.}  \label{tab:smfoo}
\begin{center}
  \begin{tabular}{|c|c|c|} \hline
   Species & \bf Mean & \bf Std.~Dev. \\ \hline
    1 & 3.4 & 1.2 \\
    2 & 5.4 & 0.6 \\ \hline
  \end{tabular}
\end{center}
\end{table}

\bibliographystyle{siamplain}
\bibliography{thesisrefs}